\documentclass[bj,preprint]{imsart}

\usepackage{amsmath,amsthm}
\RequirePackage[colorlinks,citecolor=blue,urlcolor=blue]{hyperref}

\usepackage{scrextend}

\usepackage{chngcntr}
\usepackage{apptools}
\AtAppendix{\counterwithin{lem}{section}}

\usepackage{longtable}

\usepackage{afterpage}
\usepackage{mathrsfs}
\usepackage{amssymb}
\usepackage{appendix}
\usepackage{graphicx}
\usepackage{amsfonts}
\usepackage{latexsym}
\usepackage{pdfsync}
\usepackage{subfigure}
\usepackage{color}
\usepackage{tabularx}
\usepackage{array}
\usepackage{float}
\usepackage{hhline}
\usepackage[ruled]{algorithm2e}
\usepackage{multirow}
\usepackage{arydshln}
\usepackage{enumerate}\usepackage[inline]{enumitem}
\usepackage{mathtools}
\usepackage{bbm}
\usepackage[dvipsnames]{xcolor}
\usepackage{pifont}
\usepackage{makecell}

\usepackage{relsize}

\usepackage{caption}

\usepackage{hyphenat}

\usepackage{tikz}
\usetikzlibrary{positioning}
\usepackage{textgreek}

\startlocaldefs

\DeclareMathOperator{\diam}{diam}

\DeclareMathOperator*{\argmin}{arg\,min}

\newcommand{\commentout}[1]{}
\newcommand{\ra}[1]{\renewcommand{\arraystretch}{#1}}

\newcommand{\N}{\mathbb{N}}
\newcommand{\Z}{\mathbb{Z}}
\newcommand{\R}{\mathbb{R}}

\let\S\relax \newcommand{\S}{\mathbb{S}}

\newcommand{\spn}{\operatorname{span}}

\newcommand{\supp}{\operatorname{supp}}

\let\P\relax\newcommand{\P}{\mathbb{P}}
\newcommand{\E}{\mathbb{E}}
\newcommand{\var}{\mathrm{Var}}
\newcommand{\Var}{\mathrm{Var}}
\newcommand{\cov}{\mathrm{Cov}}
\newcommand{\Cov}{\mathrm{Cov}}

\newcommand{\cc}{\mathcal{C}}
\newcommand{\mm}{\mathcal{M}}

\newcommand{\la}{\lambda}
\newcommand{\eps}{\varepsilon}

\newcommand{\kk}{\mathcal{K}}
\newcommand{\K}[1]{\kk_{#1}}
\newcommand{\hh}{\mathcal{H}}
\newcommand{\bb}{\mathcal{B}}

\newcommand{\hrho}{\widehat{\rho}}

\newcommand{\Ijkv}{{I_{j,k|v}}}

\newcommand{\Ijk}[3]{{I_{#1,#2|#3}}}

\newcommand{\hfjkv}{\widehat{f}_{j,k|v}}

\newcommand{\hfjv}{\widehat{f}_{j|v}}

\newcommand{\yjkv}{{\overline{{y}}}_{j,k|v}}
\newcommand{\hyjkv}{\widehat{{y}}_{j,k|v}}

\newcommand{\hcovjkv}{\widehat{q}_{j,k|v}}
\newcommand{\hvarjkv}{\widehat{s}_{j,k|v}}

\newcommand{\xjk}[3]{{\overline{x}_{#1,#2|#3}}}

\newcommand{\xjkv}{\overline{x}_{j,k|v}}
\newcommand{\hxjkv}{\widehat{x}_{j,k|v}}

\newcommand{\hmjkv}{\widehat{\beta}_{j,k|v}}

\newcommand{\hv}{{\widehat{v}}}

\newcommand{\hf}{\widehat{f}}

\newcommand{\Ijkhv}{{I_{j,k|\hv}}}

\newcommand{\hfjkhv}{\widehat{f}_{j,k|\hv}}
\newcommand{\hfjhv}{\widehat{f}_{j|\hv}}
\newcommand{\hFjhv}{\widehat{F}_{j|\hv}}
\newcommand{\hFjhvl}{\widehat{F}_{j|\hvc\el}}

\newcommand{\yjkhv}{{\overline{{y}}}_{j,k|\hv}}
\newcommand{\hyjkhv}{\widehat{{y}}_{j,k|\hv}}

\newcommand{\hcovjkhv}{\widehat{q}_{j,k|\hv}}
\newcommand{\hvarjkhv}{\widehat{s}_{j,k|\hv}}

\newcommand{\xjkhv}{\overline{x}_{j,k|\hv}}
\newcommand{\hxjkhv}{\widehat{x}_{j,k|\hv}}

\newcommand{\mujk}[2]{{\mu}_{#1,#2}}
\newcommand{\hmu}{{\widehat{\mu}}}
\newcommand{\hmujk}[2]{{\hmu}_{#1,#2}}

\newcommand{\hmjkhv}{\widehat{\beta}_{j,k|\hv}}

\newcommand{\hvl}{{\hv}_l}

\newcommand{\hvjk}[2]{{\hv}_{#1,#2}}

\newcommand{\vjk}[2]{{v}_{#1,#2}}

\newcommand{\II}{\mathbbm{1}}

\newcommand{\bIjkl}[3]{{E_{#1,#2|#3}}}

\newcommand{\hfjkl}{\widehat{f}_{j,k|l}}
\newcommand{\hfjl}{\widehat{f}_{j|l}}
\newcommand{\hFjl}{\widehat{F}_{j|l}}

\newcommand{\fjv}{f_{j|v}}

\newcommand{\ar}{{r}}

\newcommand{\Si}{{\Sigma}}
\newcommand{\hSi}{\widehat{\Si}}
\newcommand{\tSi}{\widetilde{\Si}}
\newcommand{\tSijk}[2]{{\tSi_{#1,#2}}}

\let\j\relax\newcommand{\j}{{j}}
\let\k\relax\newcommand{\k}{{k}}
\newcommand{\el}{{l}}

\newcommand{\h}{{h}}
\renewcommand{\L}{{L}}

\newcommand{\Ci}{{S}}
\newcommand{\Cijk}[2]{{{\Ci}_{#1,#2}}}

\newcommand{\hSijk}[2]{{\hSi}_{#1,#2}}

\newcommand{\Sijk}[2]{{\Si}_{#1,#2}}

\newcommand{\hM}{{\widehat{M}}}

\newcommand{\hMc}[1]{{\hM}_{#1}}

\newcommand{\Mc}[1]{{M}_{#1}}

\newcommand{\hV}{{\widehat{V}}}

\newcommand{\hVc}[1]{{\hV}_{#1}}

\newcommand{\hvc}[1]{{\hv}_{#1}}

\newcommand\Item[1][]{%
  \ifx\relax#1\relax  \item \else \item[#1] \fi
  \abovedisplayskip=0pt\abovedisplayshortskip=0pt~\vspace*{-1.7\baselineskip}}
  
\makeatletter
\newcommand{\myitem}[1]{%
\item[#1]\protected@edef\@currentlabel{#1}%
}
\makeatother

\newcommand{\new}[1]{#1}
\newcommand{\newnew}[1]{#1}

\theoremstyle{definition}

\theoremstyle{remark}

\theoremstyle{plain}
\newtheorem{lem}{Lemma}
\newtheorem{prop}{Proposition}
\newtheorem*{prop*}{Proposition}
\newtheorem{thm}{Theorem}
\newtheorem*{thm*}{Theorem}

\endlocaldefs

\allowbreak

\allowdisplaybreaks

\begin{document}

\begin{frontmatter}

\title{Conditional regression for single-index models}
\runtitle{Conditional regression for single-index models}


\begin{aug}
\author{\fnms{Alessandro} \snm{Lanteri}\thanksref{a} \ead[label=e1]{alessandro.lanteri@unimi.it}}, 
\author{\fnms{Mauro} \snm{Maggioni}\thanksref{b}\ead[label=e2]{mauromaggionijhu@icloud.com}}
\and
\author{\fnms{Stefano} \snm{Vigogna}\thanksref{c}*%
\ead[label=e3]{vigogna@dibris.unige.it}%
}

\address[a]{DEMM, Universit\`a degli Studi di Milano, Milano, Italy.
\printead{e1}
}

\address[b]{Department of Mathematics and Department of Applied Mathematics \& Statistics, Johns Hopkins University, Baltimore, USA.
\printead{e2}
}

\address[c]{MaLGa, DIBRIS, Universit\`a degli Studi di Genova, Genova, Italy.
\printead{e3}
}

\runauthor{A. Lanteri, M. Maggioni and S. Vigogna}

\end{aug}

\begin{abstract}
The single-index model is a statistical model for intrinsic regression
where responses are assumed to depend on a single yet unknown linear combination of the predictors,
allowing to express the regression function as $ \E [ Y | X ] = f ( \langle v , X \rangle ) $
for some unknown \emph{index} vector $v$ and \emph{link} function $f$.
Conditional methods provide a simple and effective approach to estimate $v$
by averaging moments of $X$ conditioned on $Y$, but
depend on parameters whose optimal choice is unknown and
do not provide generalization bounds on $f$.
In this paper we propose a new conditional method
converging at $\sqrt{n}$ rate under an explicit parameter characterization.
Moreover, we prove that polynomial partitioning estimates achieve the $1$-dimensional min-max rate for regression of H\"older functions
when combined to any $\sqrt{n}$-convergent index estimator.
Overall this yields an estimator for dimension reduction and regression of single-index models that attains statistical optimality in quasilinear time.
\end{abstract}

\begin{keyword}[class=MSC]
\kwd[primary ]{62G05}
\kwd[; secondary ]{62G08}
\kwd{62H99}
\end{keyword}

\begin{keyword}
\kwd{Single-index model}
\kwd{dimension reduction}
\kwd{nonparametric regression}
\kwd{finite-sample bounds}
\end{keyword}

\end{frontmatter}

\section{Introduction}

Consider the standard regression problem
of estimating a function $ F : \R^d \to \R $ from $n$ samples $ \{(X_i,Y_i)\}_{i=1}^n $,
where the $X_i$'s are independent realizations of a predictor variable $X\in\mathbb{R}^d$,
\begin{equation}
Y_i=F(X_i)+\zeta_i\,, \qquad  i = 1,\dots,n \,,
\label{e:regressionDef}
\end{equation}
and the $\zeta_i$'s are realizations, independent among themselves and of the $X_i$'s, of a random variable $\zeta$ modeling noise.
Under rather general assumptions on $\zeta$ and the distribution $\rho$ of $X$, if we only know that $F$ is $s$-H\"older regular (and, say, compactly supported), it is well-known that the min-max nonparametric rate for estimating $F$ in $L^2(\rho)$ is $n^{-s/(2s+d)}$ \cite{Gyorfi}.
This is an instance of the \emph{curse of dimensionality}: the rate slows down dramatically as the dimension $d$ increases.
Many regression models have been introduced throughout the decades to circumvent this phenomenon; see, for example, the classical reference \cite{Stone82}.
When the covariates are intrinsically low-dimensional, concentrating on an unknown low-dimensional set,
several estimators have been proved to converge at rates that are optimal with respect to the intrinsic dimension \cite{bickel2007,NIPS2011_4455,NIPS2013_5103,liao2016learning,GMRARegression}.
In other models, the domain may be high-dimensional, but the function itself is assumed to depend only on a small number of features.
A classical case is the so-called {\em{single-index model}}, where $F$ has the structure
\begin{equation}
F(x)=f(\langle v,x\rangle)
\label{e:SIMdef}
\end{equation}
for some {\em{index vector}} $ v \in \R^d $ (that we may assume unitary without loss of generality) and {\em{link function}} $f:\mathbb{R}\rightarrow\mathbb{R}$. In this context one may consider different estimation problems, depending on whether $f$ is known (e.g. in logistic regression) or both $f$ and $v$ are unknown. We are interested in the latter case.
Clearly, if $v$ was known we could learn $f$ by solving a $1$-dimensional regression problem, which may be done efficiently for large classes of functions $f$.
So the question is: what is the price to pay for not knowing $v$?

It was conjectured in \cite{Stone82} that the min-max rate for regression of single-index models is $ n^{-s/(2s+1)} $,
that is, the min-max rate for univariate functions:
no statistical cost would have to be paid.
This rate was achieved for pointwise convergence with kernel estimators in \cite[Theorem 3.3]{ADE89} and \cite[Section 2.5]{horowitz1998semiparametric},
observing that the index can be learned at the parametric rate $ n^{-1/2} $.
Based on these results or on similar heuristics, a wide part of literature focused on index estimation, setting aside the regression problem.
From this perspective, the main point is that
the estimation of the index $v$ can be carried out at parametric rate
in spite of the unknown nonparametric nonlinearity $f$.
\new{
The existence of an estimator converging in $L^2(\rho)$ at rate $ n^{-s/(2s+1)} $
was established in \cite[Corollary 22.1]{Gyorfi}.
A proof that $ n^{-s/(2s+1)} $ is also a lower min-max rate of convergence (in $L^2(\rho)$),
thus completing the proof of Stone's conjecture,
can be found in \cite[Theorem 2]{Aggr}.
}

Granted that the estimation of the index does not entail additional \emph{statistical} costs (in terms of regression rates),
a different but no less important problem is
determining the \emph{computational} cost (\new{expressed as the number of required elementary operations)} to implement a statistically optimal estimator for the single-index model.
The rate in \cite[Corollary 22.1]{Gyorfi} is obtained by a least squares joint minimization over $v$ and $f$,
but no executable algorithm is provided.
In \cite{Aggr} it was proposed to aggregate local polynomial estimators on a lattice of the unit sphere,
yielding an adaptive, universal min-max estimator, although at the expense of a possibly exponential number of operations $\Omega(n^{(d-1)/2})$. While a heuristic faster algorithm is therein also proposed, its statistical effectiveness was not proved.

Several other methods for the estimation of $v$ or $f$ were developed over the years.
A first category includes semiparametric methods based on maximum likelihood estimation \cite{ICHIMURA199371,hardle1993,Delecroix1997Efficient,delecroix1999,DELECROIX2006730,Carroll97,Carroll98}.
M-estimators produce $\sqrt{n}$-consistent index estimates under general assumptions, but their implementation is cumbersome and computationally demanding,
in that depends on sensitive bandwidth selections for kernel smoothing
and relies on high-dimensional joint optimization.
An attempt at avoiding the data sparsity problem was made by \cite{cui2011},
which proposed a fixed-point iterative scheme only involving $1$-dimensional nonparametric smoothers.
Direct methods such as the average derivative estimation (ADE \cite{Stoker86consistentestimation,ADE89})
estimate the index vector exploiting its proportionality with the derivative of the regression function.
Early implementations of this idea suffer from the curse of dimensionality
due to kernel estimation of the gradient,
while later iterative modifications \cite{hristache2001} provide $\sqrt{n}$-consistency under mild assumptions,
yet not eliminating the computational overhead.
More recently, Isotron \cite{kalai2009the,Isotron} and SILO \cite{SILO}
achieved linear complexity,
but the proven regression rate, even if independent of $d$, is not min-max
(albeit SILO focuses on the $ n \ll d $ regime,
rather than the limit $n\rightarrow\infty$ as here and most past work).
In a different yet related direction,
the even more recent \cite{BachConvexNeural} showed that convex neural networks
can adapt to a large variety of statistical models, including single-index;
however, they do not match the optimal learning rate (even for the single-index case),
and at the same time do not have associated fast algorithms.

Meanwhile, a line of research addressed \emph{sufficient dimension reduction} \cite{li2018sufficient} in the more general \emph{multi-index model} (or a slight extension thereof),
where $F$ depends on multiple $ k < d $ index directions spanning an unknown index subspace.
Along this thread we can find
nonparametric methods extending ADE to multiple indices,
such as structural adaptation (\cite{SA,dalalyan2008}),
the outer product of gradients (OPG \cite{MAVE})
and the minimum average variance estimation (MAVE \cite{MAVE,Xia2006}).
Alternatively, \textit{conditional methods}
derive their estimates from statistics of the conditional distribution of the explanatory variable $X$ given the response variable $Y$.
Prominent examples are sliced inverse regression (SIR \cite{DuanLi1,SIR}),
sliced average variance estimation (SAVE \cite{SAVE}),
simple contour regression (SCR \cite{CR}) and its generalizations (e.g. GCR \cite{CR}, DR \cite{DR}).
Conditional methods are appealing for several reasons.
Compared to nonparametric methods, their implementation is straightforward,
consisting of noniterative computation of ``sliced'' empirical moments
and having only one ``slicing'' parameter to tune.
Moreover, they are computationally efficient and simple to analyze,
enjoying $\sqrt{n}$-consistency and, in most cases,
complexity linear in the sample size and quadratic in the ambient dimension.
On the downside, this comes in general at the cost of stronger distributional assumptions,
and no theoretically optimal choice of the slicing parameter is known \cite[p. 75]{MR3211755}.
Moreover, while conditional methods offer a provable, efficient solution for sufficient dimension reduction,
they do not address the problem of estimating the link function on the estimated index space.
We summarize the key properties for these and other aforementioned techniques in Table \ref{tab:market} below.

\begin{center}
\begin{table}[h] \label{tab:market}
\caption{
Proven rate \new{in $L^2(\rho)$} (up to log factors) and computational cost of several methods for index estimation and/or regression in single-index models, together with salient assumptions on the model.
}

\label{tab:market}
\renewcommand{\arraystretch}{1.2}
\resizebox{\columnwidth}{!}{%
\begin{tabular}{c  c  c  c  c  c  c  c }
\cline{2-8}
& \multicolumn{4}{c }{\textbf{Performance}} & \multicolumn{3}{c}{\textbf{Assumptions}}
\\
\cline{2-8}
& \multicolumn{2}{c }{\textbf{Proven rate}} & \multicolumn{2}{c }{\textbf{Computational cost}} & \multirow{2}{*}{$\pmb{X}$} & \multirow{2}{*}{$\pmb{f}$} & \multirow{2}{*}{$\pmb{\zeta}$}
\\
& $\pmb{\hv}$ & $\pmb{\hf}$  & $\pmb{\hv}$ & $\pmb{\hf}$ & & &
\\
\hline
\multicolumn{1}{ l }{\textbf{SIR}} \cite{SIR} & $n^{-1/2}$ & $-$ & $d^2n\log n$ & $-$ & linear $ \E [ X | v^TX ] $ & N/A & N/A
\\
\hline
\multicolumn{1}{ l }{\textbf{SAVE}} \cite{SAVE} & $n^{-1/2}$ & $-$ & $d^2n\log n$ & $-$ & \makecell{ linear $ \E [ X | v^TX ] $, \\ const $ \Cov [ X | v^TX ] $ } & N/A & N/A
\\
\hline
\multicolumn{1}{ l }{\textbf{SCR}} \cite{CR} & $n^{-1/2}$ & $-$ & $d^2n^2\log n$ & $-$ & \makecell{ linear $ \E [ X | v^TX ] $, \\ const $ \Cov [ X | v^TX ] $ } & \makecell{stochastically\\monotone} & \makecell{decreasing\\density of $\zeta - \widetilde{\zeta}$}
\\
\hline
\multicolumn{1}{ l }{\textbf{DR}} \cite{DR} & $n^{-1/2}$ & $-$ & $d^2n\log n$ & $-$ & \makecell{ linear $ \E [ X | v^TX ] $, \\ const $ \Cov [ X | v^TX ] $ } & N/A & N/A
\\
\hline
\multicolumn{1}{ l }{\textbf{ADE}} \cite{hristache2001} & $n^{-1/2}$ & $-$ & $ d^2n^2\log n $ & $-$ & $\cc^0$ positive density & $ \cc^2 $ & Gaussian
\\
\hline
\multicolumn{1}{ l }{\textbf{rMAVE}} \cite{Xia2006} & $n^{-1/2}$ & N/A & \multicolumn{2}{c }{$d^2n^2$ per iteration} & \multicolumn{1}{c }{\makecell{ $v^TX$ has $\cc^3$ density, \\ $ \E |X|^6 < \infty $ } } & \multicolumn{1}{c }{ $\cc^3$ } & $ \E|Y|^3 < \infty $
\\
\hline
\multicolumn{1}{ l }{\textbf{Aggregation}} \cite{Aggr} & $-$ & $n^{-\tfrac{s}{2s+1}}$ & \multicolumn{2}{c }{$(n\log n)^d$} & \makecell{compact supported\\lower bounded density} & $\cc^s$ & $\sigma(X)\mathcal{N}(0,1)$
\\
\hline
\multicolumn{1}{ l }{\textbf{SlIsotron}} \cite{Isotron} & N/A & $n^{-1/6}$ & \multicolumn{2}{c }{$(\tfrac{n}{d\log n})^{1/3}dn\log n$} & bounded & \makecell{monotone,\\Lipschitz} & bounded
\\
\hline
\multicolumn{1}{ l }{\textbf{SILO}} \cite{SILO} & $n^{-1/4}$ & $n^{-1/8}$ & $dn$ & $n\log n$ & Gaussian & \makecell{monotone,\\Lipschitz} & bounded
\\
\hline
\hline
\multicolumn{1}{ l }{\textbf{SVR}} & $n^{-1/2}$ & $n^{-\tfrac{s}{2s+1}}$ & $d^2n\log n$ & $n\log n$ & \makecell{ linear $ \E [ X | v^TX ] $, \\ $ \Var [ w^TX | v^TX ] \gtrsim 1 $ } & \makecell{coarsely\vspace{-2pt}\\monotone,\\ $\cc^s$} & sub-Gaussian
\\
\hline
\end{tabular}%
}
\end{table}
\end{center}

In this work we introduce a new estimator and a corresponding algorithm, called Smallest Vector Regression (SVR), that are statistically optimal and computationally efficient.
Our dimension reduction technique falls in the category of conditional methods.
Unlike existing studies for similar approaches,
we are able to provide a characterization for the parameter selection,
and bound both the index estimation and the regression errors.
Since regression is performed using standard piecewise polynomial estimates on the projected samples after and independently of the index estimation step,
our regression bounds hold conditioned to any index estimation method of sufficient accuracy.
Our analysis yields convergence by proving finite-sample bounds in high probability.
The resulting statements are stronger compared to the ones in the available literature on conditional methods,
where typically only asymptotic convergence, at most, is established.
As a side note, SVR has been empirically tested with success also in the multi-index model,
but our analysis, and therefore our exposition, will be restricted to the single-index case
(for a related analysis of the multi-index model, see \cite{10.1214/20-EJS1785}).
In summary, the contributions of this work are:
\begin{enumerate}
\item We introduce a new conditional regression method that combines accuracy, robustness and low computational cost. This method is multiscale and sheds light on parameter choices that are important in theory and practice, and are mostly left unaddressed in other techniques.
\item We prove strong, finite-sample convergence bounds, both in probability and in expectation, for the index estimate of our conditional method.
\item We prove that \new{the population statistics used to construct} conditional polynomial partitioning estimation \new{(up to degree $1$) are} H\"older continuous with respect to the index estimation. This allows to bridge the gap between a good estimator of the index subspace and the performance of regression on the estimated subspace.
\item We prove that all $\sqrt{n}$-convergent index estimation methods
lead to the min-max $1$-dimensional rate of convergence
when combined with polynomial partitioning estimates \new{(up to degree $1$)}.
\item Using the above,
we provide an algorithm for the estimation of the single-index model
with theoretical guarantees of optimal convergence in quasilinear time.
\end{enumerate}

The paper is organized as follows. In Section \ref{s:ConditionRegressionMethods} we review several conditional regression methods, and introduce our new estimator; in Section \ref{s:AnalysisConvergence} we analyze the converge of our method and establish min-max rates for regression conditioned on any sufficiently accurate index estimate; in Section \ref{s:NumericalExperiments} we conduct several numerical experiments, both validating the theory and exploring numerically aspects of various techniques that are not covered by theoretical results;
in the Appendix we collect additional proofs and technical results.

\begin{table}[H]
\captionsetup{labelformat=empty}
\caption{\normalfont \textbf{Notation}}
\label{tab:notation}
\centering
\ra{1.2}
\resizebox{\columnwidth}{!}{
\begin{tabular}{l l l l}

\textbf{symbol} & \textbf{definition} & \textbf{symbol} & \textbf{definition} \\

\hline

$C,c$ & positive absolute constants & $ \|A\| $ & spectral norm of a matrix $A$ \\

$a \lesssim b$ & $ a \le C b $ for some \new{positive absolute constant} $C$ & $ \la_i(A) $ & $i$-th largest eigenvalue of a matrix $A$ \\

$ a \asymp b $ & $a \lesssim b$ and $b \lesssim a$ & $ |I| $ & Lebesgue measure of an interval $I$ \\

$ \|u\| $ & Euclidean norm of a vector $u$ & $ \new{\#\mathcal{S}} $ & \new{cardinality of a set $\mathcal{S}$} \\

$ B(x,r) $ & Euclidean ball of center $x$ and radius $r$ & $ \II (E) $ & indicator function of an event $E$ \\

\new{$ \spn\{\mathcal{S}\} $ } & linear span of a set $\mathcal{S}$ & $\P(E)$  & probability of an event $E$ \\

\new{$ \{\mathcal{S}\}^\perp$} & orthogonal complement of a set $\mathcal{S}$ & $\E[X]$ & expectation of a random variable $X$ \\

\new{$ P_u $ } & orthogonal projection onto $ \spn\{u\} $ & $ \var[X] $ & variance of a r.v. $X$ \\

\new{$\angle(u,v)$} & angle between $\spn{u}$ and $\spn{v}$ & $ \cov[X] $ & covariance matrix of a r.v. $X$ \\

$ \langle u , v \rangle $ & inner product of vectors $u$ and  $v$ & $ X \mid Y $ & r.v. $X$ conditioned on r.v. $Y$ \\

\hline

\end{tabular}
}
\end{table}

\setcounter{table}{1}

\section{Conditional regression methods}
\label{s:ConditionRegressionMethods}

We consider the regression problem as in \eqref{e:regressionDef}, within the single-index model, with the definition and notation as in \eqref{e:SIMdef}.
When $f$ is at least Lipschitz, \eqref{e:SIMdef} implies $ \nabla F(x) \in \spn\{v\} $ for a.e. $x$;
this is the reason why we may refer to $v$ as the gradient direction.
Given $n$ independent copies $ (X_i,Y_i) $, $ i = 1,\dots,n $, of the random pair $ (X,Y) $,
we will construct estimators $\hv$ and $\hf$,
and derive separate and compound non-asymptotic error bounds in probability and expectation.
Our method is conditional in two ways:
\begin{enumerate*}[label={\arabic*)}]
\item \label{it:hv} the estimator $\hv$ is a statistic of the conditional distributions of the $X_i$'s given the $Y_i$'s (restricted in suitable intervals);
\item \label{it:hf} the estimator $\hf$ is conditioned on the estimate $\hv$.
\end{enumerate*}
Several conditional methods for step \ref{it:hv} have been previously introduced, see e.g. \cite{SIR,SAVE,CR,DR}.
Our error bounds for step \ref{it:hf} are independent of the particular method used in \ref{it:hv}, only requiring a minimal non-asymptotic convergence rate.
For these reasons, one may as well consider other existing or new methods for \ref{it:hv},
even non conditional,
and check for each one the convergence rate needed to pair it with \ref{it:hf}.

The common idea of all conditional methods is to compute statistics of the predictor $X$ conditioned on the response $Y$.
Conditioning on $Y$, one introduces anisotropy in the distribution of $X$,
forcing it to reveal the index structure through its moments,
be they means (SIR) or variances (SAVE, SCR, DR).

Before introducing SVR,
we will review the two methods that have strongest connections with ours, namely SIR and SAVE.
For consistency with SVR, we will present SIR and SAVE through a particular multiscale implementation.
This will allow to progressively define the objects SVR is built upon,
facilitating the comparison.


\subsection{Sliced Inverse Regression}
\label{s:SIR}

Sliced Inverse Regression \cite{SIR} (SIR) estimates the index vector
by a principal component analysis of a sliced empirical approximation of the inverse regression curve $ \E[X | Y] $.
Samples on this curve are obtained by slicing the range of the function
and computing sample means of the corresponding approximate level sets.
In our version of SIR, we take dyadic partitions $ \{ {\Cijk\el\h} \}_{\h=1}^{2^\el} $, $ \el \in \Z $, of a subinterval of the range of $Y$, where each $ {\Cijk\el\h} $ is an interval of length $ \asymp 2^{-\el} $.
\new{
Let
$
 n_{\el,\h} = \# \{Y_i \in \Cijk\el\h\} .
$
}
After calculating the sample mean for each slice,
\begin{equation*}
\hmujk\el\h = \frac{1}{\new{n_{\el,\h}}} \sum_i X_i \II\{Y_i \in {\Cijk\el\h} \} \,, \qquad \h = 1,\dots,2^{\el}\,,
\end{equation*}
SIR outputs $ \hvc\el $ as the eigenvector of largest eigenvalue of the weighted covariance matrix
$$
 \hMc\el = \sum_{\h} \hmujk\el\h \hmujk\el\h^T \frac{\new{n_{\el,\h}}}n \, .
$$
Note that the population limits of $\hmujk\el\h$ and $\hMc\el$ are, respectively,
$$
 \mujk\el\h = \E [ X \mid Y \in {\Cijk\el\h} ] \,, \qquad \Mc\el = \sum_\h \mujk\el\h \mujk\el\h^T \P\{ Y \in {\Cijk\el\h} \} \,.
$$

\subsection{Sliced Average Variance Estimation} \label{s:SAVE}
Sliced Average Variance Estimation \cite{SAVE} (SAVE)
generalizes SIR to second order moments.
After slicing the range of $Y$ and computing the centers $ \hmujk\el\h $'s,
it goes further and construct the sample covariance on each slice:
\begin{equation*}
   \hSijk\el\h = \frac{1}{\new{n_{\el,\h}}} \sum_i (X_i - \hmujk\el\h) (X_i - \hmujk\el\h)^T \II\{Y_i \in {\Cijk\el\h}\}\,.
\end{equation*}
Then, it averages the $ \hSijk\el\h $'s and defines $ \hv_\el $ as the eigenvector of largest eigenvalue of
\begin{equation*}
 \hSi_\el = \sum_h (I - \hSijk\el\h)^2 \frac{\new{n_{\el,\h}}}{n} \,.
\end{equation*}
The matrices $ \hSijk\el\h $ and $ \hSi_\el $ are empirical estimates of
\begin{equation*}
\Sijk\el\h = \cov [ X \mid Y \in {\Cijk\el\h} ] \,, \qquad \Si_\el = \sum_h (I - \Sijk\el\h)^2 \P \{ Y \in \Cijk\el\h \} \,.
\end{equation*}

\subsection{Smallest Vector Regression (SVR)}

This is the new method we propose here. We perform a local principal component analysis on each approximate level set obtained by multiscale slices of $Y$.
Because of the special structure \eqref{e:SIMdef},
we expect each (approximate) level set to be narrow along $v$ and spread out along the orthogonal directions.
Thus, the smallest principal component should approximate $v$.
Once we have an estimate for $v$, we can project down the $d$-dimensional samples and perform nonparametric regression of the $1$-dimensional function $f$.
The method consists of the following steps:
\begin{enumerate}[label=\textnormal{1.\alph*)}]
 \item \label{it:Clh} Construct a multiscale family of dyadic partitions of a subinterval $\Ci$ of the range of $Y$:
 $$
  \{ {\Cijk\el\h} \}_{\h=1}^{2^\el} \,, \qquad \el\in\Z \,.
 $$
 For each $l$, $ \{ \Cijk\el\h \}_h $ is a partition of $\Ci$ with $ | {\Cijk\el\h} | = |\Ci| 2^{-\el} $.
 \item \label{it:vlh} Let $\hh_\el$ be the set of $h$'s such that $\new{n_{\el,\h}} \ge 2^{-\el} n $. For $h\in \hh_\el$, let $\hvjk\el\h$ be the eigenvector corresponding to the \emph{smallest} eigenvalue of $\hSijk\el\h$.
 \item \label{it:vl} Compute the eigenvector $\hvc\el$ corresponding to the \emph{largest} eigenvalue of
 $$
  \hVc\el = \frac{1}{\sum_{\h\in \hh_\el} \new{n_{\el,\h}}} \sum_{\h\in \hh_\el} \hvjk\el\h \hvjk\el\h^T  \new{n_{\el,\h}}\, .
 $$
  \end{enumerate}
  \begin{enumerate}[label=2)]
 \item Regress $f$ using a dyadic polynomial estimator $ \hf_{j|\hvc\el} $ at scale $ j \ge 0 $ on the samples $ (\langle \hvc\el, X_i\rangle , Y_i ) $, $ i = 1,\dots,n $ (more details in Section \ref{sec:hf}). Return $ \hFjhvl(x) = \hf_{j|\hvc\el} (\langle \hvc\el,x\rangle) $.
 \end{enumerate}
 
While SVR shares step \ref{it:Clh} with SIR, it differs from SIR in step \ref{it:vlh},
where it takes conditional (co)variance statistics in place of conditional means,
and in step \ref{it:vl}, where it averages smallest-variance directions rather than means.
We may regard SAVE and SVR as two different modifications of SIR to higher order statistics,
which allows in general for better and more robust estimates (see Section \ref{s:numestv}).
The fundamental difference between SVR and SAVE is that
SVR computes local estimates of the index vector which then aggregates in a global estimate,
while SAVE first aggregates local information and then computes a single global estimate.

\new{
In our theoretical analysis we will assume $Y$ sub-Gaussian,
and choose $S$ deterministically with radius multiple of a standard deviation proxy of $|Y|$,
in order to ensure that $S$ contains most (but possibly not all) samples (see Section \ref{sec:index-bounds}).
However, a practically good choice for $S$ is the smallest closed interval containing all the observed data, namely $ [ \min_i Y_i , \max_i Y_i ] $,
in which case $S$ would be a random interval.
If $Y$ is bounded in a known interval, one may also take $S$ as that interval,
so that $S$ would be again deterministic.
}

\subsection{Conditional partitioning estimators} \label{sec:hf}
In step \ref{it:hf} we use piecewise polynomial estimators in the spirit of \cite{BCDDT1,BCDD2}: these techniques are based on partitioning the domain (here, in a multiscale fashion), and constructing a local polynomial on each element of the partition by solving a least squares fitting problem.
A global estimator is then obtained by summing the local polynomials over a certain partition (possibly using a partition of unity to obtain smoothness across the boundaries of the partition elements).
The degree of the local polynomials needed to obtain optimal rates depends on the regularity of the function, and may be chosen adaptively if such regularity is unknown.
A proper partition (or scale) is then chosen to minimize the expected mean squared error (MSE), by classical bias-variance trade-off.

In detail, given an estimated direction $\widehat v$, our step \ref{it:hf} consists of:
\begin{enumerate}[label=\textnormal{2.\alph*)}]
\item \label{it:Ijk} Construct a multiscale family of dyadic partitions of an interval $I$:
$$
  \{ I_{j,k} \}_{\k\in\K\j} \,, \qquad \j\in\Z \,.
 $$
 For each $j$, $\{ I_{j,k} \}_{\k\in\K\j}$ is a partition of $I$ with $ | I_{j,k} | = |I| 2^{-\j} $.
 \item \label{it:fjk} For each $ I_{j,k} $, compute the best fitting $m$-order polynomial
 $$
 \hfjkhv = \argmin_{\deg(p) \le m} \sum_{i} | Y_i - p(\langle \hv,X_i\rangle) |^2 \II\{\langle \hv,X_i\rangle \in I_{j,k} \} \,.
 $$
 \item \label{it:fj} Sum all \new{(truncated)} local estimates $ \hfjkhv $ over the partition $ \{I_{j,k}\}_{\k\in\K\j} $:
 $$
 \hfjhv (t) = \sum_{\k\in\K\j} \new{T_M[}\hfjkhv (t) \new{]} \II\{t \in I_{j,k}\} \,,
 $$
\new{ where $ T_M [t] = \operatorname{sign}(t) \min\{ |t| , M \} $ for some $ M \in (0,+\infty] $.}
 \end{enumerate}
 The final estimator of $F$ at scale $j$ and conditioned on $\hv$ is given by
\begin{equation*}
  \hFjhv(x) = \hfjhv (\langle \hv,x\rangle) \,.
\end{equation*}

In SVR, step \ref{it:hf} is carried out on $ \hv = \hvl $, yielding for each $\el$ a multiscale family of local polynomials $ \hfjkl $ and global estimators $ \hfjl , \hFjl $.
However, we will prove results on the performance of \ref{it:hf} also when $\hv$ is the output of any estimator with $n^{-1/2}$ probabilistic convergence rate (Theorem \ref{thm:main}).

Note that for SVR, but also for SIR and SAVE, the final estimator $ \hFjl $ depends on two scale parameters, $l$ and $j$, which may be chosen independently.
Our analysis yields optimal choices for these two scale parameters;
the scale  $2^{-\el}$ at which the direction $v$ is estimated will not be finer than the noise level,
while a possibly finer partition with $\j > \el$ may be selected to improve the polynomial fit,
allowing the estimator $ \hFjl $ to de-noise its predictions, provided that enough training samples are available (see Figure \ref{fig:reg_example}).

\begin{figure}[h]

\centering

\includegraphics[width = 0.3\linewidth]{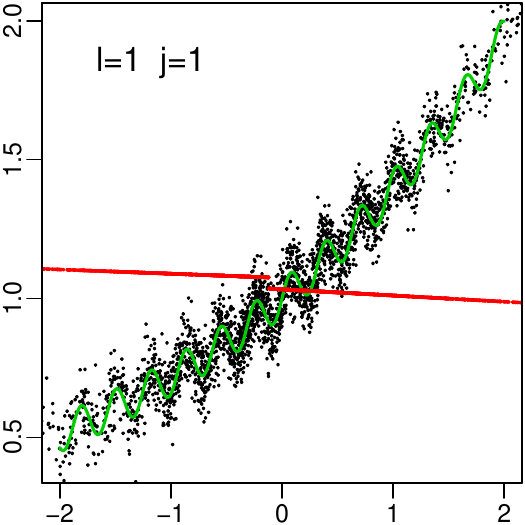}
\includegraphics[width = 0.3\linewidth]{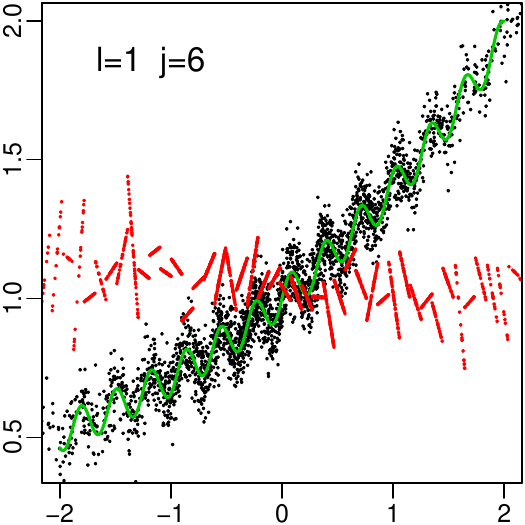}
\includegraphics[width = 0.3\linewidth]{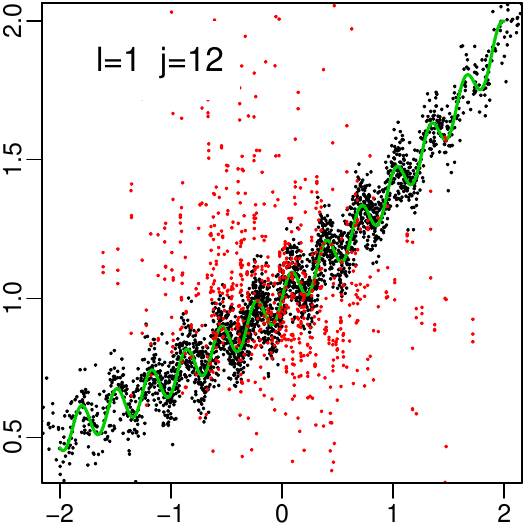} \\
\includegraphics[width = 0.3\linewidth]{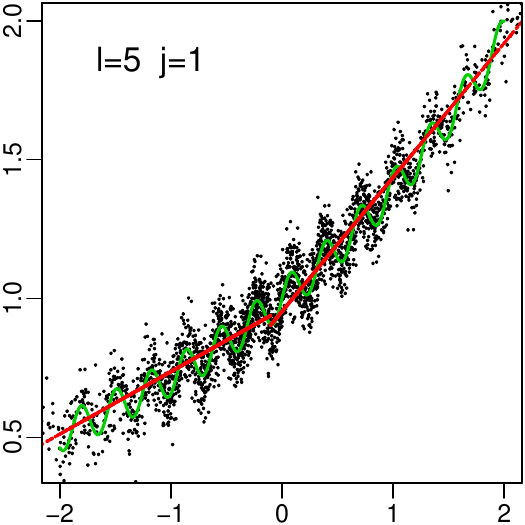}
\includegraphics[width = 0.3\linewidth]{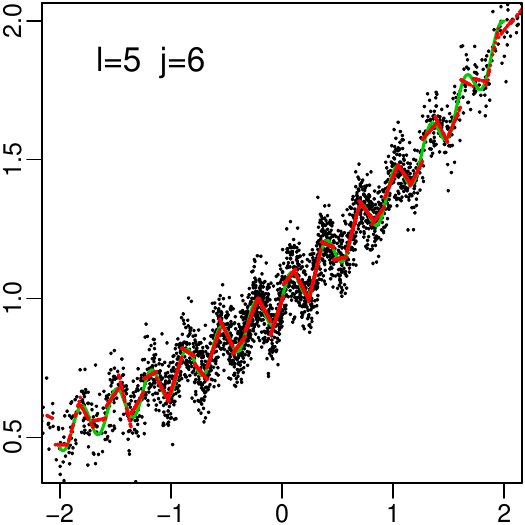}
\includegraphics[width = 0.3\linewidth]{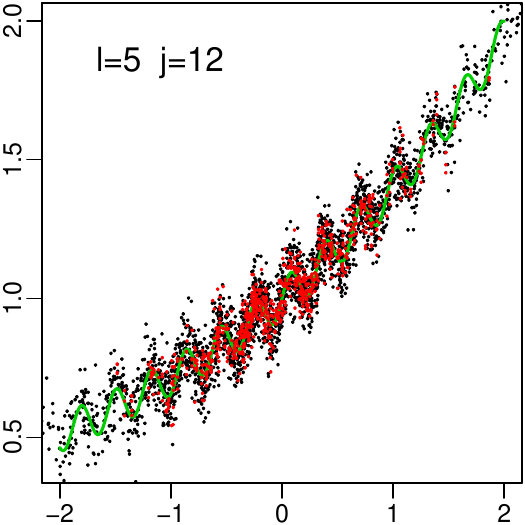}

\caption{ Local linear estimator (red) at different scales $l$, to regress the function $f$ (green) from noisy samples (black). The horizontal axis is $\langle v,x\rangle$, while of course the estimator $\hat f_{j|l}$ is a function of $\langle\hv,x\rangle$ and may appear multi-valued in $\langle v,x\rangle$.
For small $l$ (top row) the error in the estimation of the index $v$ is large, leading to poor regression estimates regardless of the regression scale $j$. For larger $l$ (bottom row) a good accuracy for the index vector $v$ is achieved, and the estimator is able to approximate the function even below the noise level and the non-monotonicity scale (e.g. for $j=6$); overfitting occurs for $j$ too large (e.g. $j=12$ in this case).\label{fig:reg_example}}

\end{figure}

\new{
For theoretical purposes, we will assume $X$ sub-Gaussian,
and take $I$ as a deterministic interval with radius
proportional to a standard deviation proxy of $\|X\|$ (see Section \ref{sec:regression}).
Again, the practitioner may choose $I$ as the random interval $ [ \min_i \langle \hv , X_i \rangle , \max_i \langle \hv , X_i \rangle ] $,
or as the projection of the support of $X$, if this is known.
}
\new{The truncation level $M$ is ideally $ \| F \|_\infty ( = \| f \|_\infty ) $, or a proxy thereof.}

We report below the complete sequence of steps run by SVR.
The time complexity of the algorithm is shown in Table \ref{tab:cost}.
Note that \ref{it:fj} has only an evaluation cost, i.e. $\hfjhv$ does not need to be constructed, but only evaluated.

\begin{algorithm}[h]
\caption{SVR} \label{alg:SVR}

 \SetKwInOut{Input}{Input}
 \SetKwInOut{Output}{Output}
 
 \Input{
  samples $ \{ (X_i,Y_i) \}_{i=1}^n \subset \R^d \times \R $, intervals $S,I$, polynomial degree $ m \in \N $\new{, truncation level $ M \in (0,\infty] $.}
 }
 
 \Output{
 $\hvl$ estimate of $v$, $\hfjl$ estimate of $f$.
 }
 
 \hrulefill
 
 \begin{enumerate}[label=\footnotesize\textbf{1.\alph*)}]
 
 \item construct $ \{{\Cijk\el\h}\}_{\el,\h} $, dyadic decomposition of $\Ci$;
 \item 
 for all $ h \in \hh_\el = \{ h : n_{\el,\h} \ge 2^{-\el} n \} $,
 where  $ n_{\el,\h} = \#\{ Y_i \in \Cijk \el\h \} $, compute \\
 $ \hmujk\el\h = \frac{1}{n_{\el,\h}} \sum_i X_i \II\{Y_i \in {\Cijk\el\h} \}$, \\
 $ \hSijk\el\h = \frac{1}{n_{\el,\h}} \sum_i (X_i - \hmujk\el\h) (X_i - \hmujk\el\h)^T \II\{Y_i \in {\Cijk\el\h} \}$, \\
 $\hvjk\el\h$, the eigenvector of $\hSijk\el\h$
 corresponding to the smallest eigenvalue;
 \item compute $\hvl$, the eigenvector of $ \hVc\el = \frac{1}{\sum_{\h\in\hh_\el}n_{\el,\h}} \sum_{\h\in\hh_\el} \hvjk\el\h \hvjk\el\h^T n_{\el,\h}$ corresponding to the largest eigenvalue;
 \end{enumerate}
 
 \begin{enumerate}[label=\footnotesize\textbf{2.\alph*)}]
 \item construct $ \{I_{j,k} \}_{\j,\k} $, dyadic decomposition of $I$;  
 \item compute $ \hfjkl = \argmin_{\deg(p) \le m} \sum_{i} | Y_i - p(\langle \hvl, X_i\rangle) |^2 \II\{\langle \hvl, X_i\rangle \in I_{j,k}\} $;
 \item define $ \hfjl (t) = \sum_{\k} \new{T_M [}\hfjkl (t)\new{]} \II\{t \in I_{j,k}\} $. 
 \end{enumerate}

\end{algorithm}

\begin{table}[h]
\caption{Computational cost breakdown for SVR.}
\label{tab:cost}
\ra{1.5}
\centering
\begin{tabular}{l l l }
 \cline{2-3}
 & \textbf{task} & \textbf{computational cost} \\
 \cline{2-3}

 \ref{it:Clh} & dyadic decomposition of the range & $ O(n\log n) $ \\

  \ref{it:vlh} & PCA on level sets & $ O(d^2 n\log n) $ \\

 \ref{it:vl} & PCA of local directions & $ O(d^2 n \log n) $ \\
 \cline{2-3}
 \ref{it:Ijk} & dyadic decomposition of the domain & $ O(n \log n) $ \\
 \ref{it:fjk}
 & $m$-order polynomial regression & $ O(m^2 n \log n) $ \\
 \cline{2-3}
 & total & $ O((d^{2} + m^2) n \log n) $ \\
 \cline{2-3}
\end{tabular}
\end{table}

\section{Analysis of convergence}
\label{s:AnalysisConvergence}

\new{
We next state and discuss several assumptions that are going to be useful for our theoretical analysis.
We point out that, when stating an assumption, such an assumption it is not intended to be implicitly assumed once for all throughout the paper.
Rather, in each result we will explicitly write which assumptions are being used,
referring to the corresponding labels.
}
We start by collecting a few basic requirements on the distributions of $X$, $Y$ and $\zeta$:
\begin{enumerate}[label = \textnormal{(X)}]
\item \label{X} $X$ has sub-Gaussian distribution with variance proxy $R^2$.
\end{enumerate}
\begin{enumerate}[label = \textnormal{(Y)}]
\item \label{Y} $Y$ has sub-Gaussian distribution with variance proxy $R^2$.
\end{enumerate}
\begin{enumerate}[label = \textnormal{(Z)}]
\item \label{Z} $ \zeta $ is sub-Gaussian with variance proxy $\sigma^2$.
\end{enumerate}
\ref{X}, \ref{Y} and \ref{Z} are standard assumptions in regression analysis {\it tout court}.
The following is instead typical of single-index models:
\begin{enumerate}[label = \textnormal{(LCM)}]
\item \label{LCM} \new{$ \E [ X \mid \langle v , X \rangle ] $ is linear in $ \langle v , X \rangle $.}
\end{enumerate}
\ref{LCM} is commonly referred to as
the \emph{linear conditional mean} assumption \cite[Condition 3.1]{SIR}.
For standardized $X$, it is equivalent to requiring
\new{ $ \E [ X \mid \langle v , X \rangle ] = P_v X $} \cite[Lemma 1.1]{li2018sufficient}.
Every elliptical distribution
satisfies \ref{LCM} for every $v$ \cite[Corollary 5]{CAMBANIS1981},
and conversely \cite{EATON1986}.
While it does introduce some symmetry, it is less restrictive than it may seem.
It has been shown to hold approximately in high dimension,
where most low-dimensional projections are nearly normal \cite{diaconis1984,hall1993}.
\ref{LCM} is introduced to ensure that $\hv$ is an unbiased estimate of $v$ \cite[Theorem 1]{10.2307/2669934}.
As an alternative to \ref{LCM}, one can rely on the stronger assumption
\begin{enumerate}[label = \textnormal{(SMD)}]
\myitem{\textnormal{(SMD)}} \label{SMD}
for every $ w \in \spn\{v\}^\perp $,
$ \langle w , X \rangle \mid \langle v , X \rangle $ has symmetric distribution (i.e. has same distribution as $ - \langle w , X \rangle \mid \langle v , X \rangle $).
\end{enumerate} 
\new{
The restriction \ref{SMD} to \emph{symmetrical marginal distributions}
will be in fact our working assumption.
Assuming \ref{SMD} rather than \ref{LCM} is for us a purely technical choice,
dictated by our goal of disentangling the dependence of the SVR estimator $\hv_\el$ on the slicing parameter $\el$.
To be more specific, our motivation stems from the application of the Bernstein inequality to concentrate the sample statistics to their population limits.
Since $X$ and $Y$ will be assumed to be sub-Gaussian, and therefore possibly unbounded,
we may not use the standard Bernstein inequality for bounded variables \cite[Lemma 2.2.9]{vaart1996},
and we should rather use the sub-exponential version \cite[Lemma 2.2.11]{vaart1996}.
As it turns out, the slicing parameter would be encoded
in the sub-Gaussian norm of $X$ conditioned on a slice,
which is however an implicit quantity, difficult to characterize in general.
On the other hand, Bernstein for bounded variables, although less general, provides a sharper bound,
with probability scaling with the variance rather than the sub-Gaussian norm.
Consequently, it allows to encode the slicing parameter in an explicit variance term.
But in order to apply Bernstein for bounded variables,
and thus obtain a more interpretable bound,
we need to condition the statistics of interest in suitable balls of constant radius (see Section \ref{sec:index-bounds}).
Such conditioning would in general break \ref{LCM}, but not \ref{SMD},
hence the need to assume \ref{SMD}.
}

\newnew{
Although stronger than \ref{LCM}, the condition \ref{SMD} is still satisfied by all elliptical distributions,
as well as other types of symmetric distributions along the index direction.
While ellipticity serves often as a simplified model for the analysis of conditional methods,
it is not necessarily crucial when it comes to applications.
In particular, the performance of SVR seems to be robust to even considerable deviations from such assumptions,
as we illustrate experimentally in Section \ref{s:numestv}.
}


In addition to \ref{LCM} or \ref{SMD}, second order methods usually require the so-called \emph{constant conditional variance} assumption \cite[p. 2117]{SAVE}:
\begin{enumerate}[label = \textnormal{(CCV)}]
\item \label{CCV} $ \cov[ X \mid \langle v , X \rangle ] $ is nonrandom.
\end{enumerate}
\new{For standardized $X$}, and assuming \ref{LCM}, \ref{CCV} is equivalent to $ \cov[ X \mid \langle v , X \rangle ] = \new{P_{v^\perp}} $ \cite[Corollary 5.1]{li2018sufficient}.
\ref{CCV} is true for the normal distribution \cite[Proposition 5.1]{li2018sufficient}, and again approximately true in high dimension \cite{diaconis1984,hall1993}.
Some care is required when assuming both \ref{LCM} and \ref{CCV}:
imposing \ref{LCM} for every $v$  is equivalent to assuming spherical symmetry \cite{EATON1986},
and the only spherical distribution satisfying \ref{CCV} is the normal distribution \cite[Theorem 7]{Kelker}.
For this reason we will introduce a relaxation of \ref{CCV} in the next subsection.

We present separately bounds on the estimation of $v$ in the next Section \ref{sec:index-bounds},
and on the regression of $f$ in Section \ref{sec:regression}.
Our main result, Theorem \ref{thm:main}, will give in particular a near-optimal high probability bound on the SVR estimator for $F$, in the form
$$ \left( \E_X [ | \widehat{F}(X) - F(X) |^2 ] \right)^{1/2} \le K ( d\sqrt{\log d} + d^s) (\log n)^{s/2} \left( \frac{\log n}{n} \right)^{\frac{s}{2s+1}} $$
\new{where $F$ is $s$-H\"older continuous} with $ s \in [1/2,2] $,
and $K$ is a constant independent of $n$ and $d$.

\subsection{Index estimation bounds} \label{sec:index-bounds}

\new{
The goal of this section is to obtain finite-sample bounds on the SVR index estimate $\hv_\el$.
In order to obtain an explicit characterization with respect to the parameter $\el$,
we introduce a variant of $\hv_\el$ based on bounded statistics.
As discussed earlier in Section \ref{s:AnalysisConvergence}, similar but less interpretable bounds for SVR can be proved for unbounded statistics.
}

\new{
Consider the event
\begin{equation*} 
\bb = \{ \|X\| \le C_X \sqrt{d} R , |Y| \le C_Y R \}
\end{equation*}
for some $ C_X , C_Y \ge 1 $ independent of $n$ and that will be considered as fixed constant from now on.
We define the following bounded versions of $ \mujk\el\h $, $ \Sijk\el\h $:
$$
 \mujk\el\h^b = \E [ X \mid Y \in {\Cijk\el\h} , \bb ]  ,
 \qquad
 \Sijk\el\h^b = \cov [ X \mid Y \in {\Cijk\el\h} , \bb ] .
$$
Given the event
\begin{equation*} 
 \bb_i = \{ \|X_i\| \le C_X \sqrt{d} R , |Y_i| \le C_Y R \} , 
\end{equation*}
we define
$$
 n^b = \sum_i \II (\bb_i) , \qquad n_{\el,\h}^b = \sum_i \II ( \{ Y_i \in \Cijk\el\h \} \cap \bb_i ) .
 $$
Note that $n^b$ is a random number,
but assuming \ref{X} and \ref{Y}
it is larger than a constant fraction of $n$, with high probability (see Lemma \ref{lem:sub-gauss-ball}).
The sample counterparts of $\mujk\el\h^b$, $\Sijk\el\h^b$ are
$$
\hmujk\el\h^b = \frac{1}{n_{\el,\h}^b} \sum_i X_i \II ( \{Y_i \in {\Cijk\el\h} \} \cap \bb_i ) ,
\ \ 
 \hSijk\el\h^b = \frac{1}{n_{\el,\h}^b} \sum_i (X_i - \hmujk\el\h^b) (X_i - \hmujk\el\h^b)^T \II ( \{Y_i \in {\Cijk\el\h} \} \cap \bb_i ) .
$$
We denote by $ \vjk\el\h^b $ and $ \hvjk\el\h^b $
the eigenvectors of smallest eigenvalue of $ \Sijk\el\h^b $ and $ \hSijk\el\h^b $, respectively.
Finally,
let $ \hh_\el^b = \{ h : n_{\el,\h}^b \ge 2^{-\el} n^b \} $.
We define $ \hvc \el^b $ as the eigenvector of largest eigenvalue of
$$
\hVc\el^b = \frac{1}{\sum_{\h\in\hh_\el^b} n_{\el,\h}^b} \sum_{\h\in\hh_\el^b} \hvjk\el\h^b (\hvjk\el\h^b)^T  n_{\el,\h}^b\, .
$$
}
In accordance with $\bb$,
we pick the interval $\Ci$ in SVR as
\new{$$
\Ci = [ - C_Y R , C_Y R ] .
$$}
Furthermore, we will be assuming lower bounded conditional variance
on the distribution conditioned on $\bb$:

\begin{enumerate}[label=\textnormal{(LCV)}]
\item \label{LCV} There is $\alpha\ge1$ such that $ \var [ \langle w,X\rangle\mid \langle v,X\rangle , \bb ] \ge R^2 / \alpha $ almost surely for all $ w \in \spn\{v\}^\perp  \cap \S^{d-1} $.
\end{enumerate}
This assumption can be seen as a relaxation of the standard \ref{CCV} we mentioned in Section \ref{s:AnalysisConvergence}.
\new{It should be contrasted with \ref{X}:
while \ref{X} imposes a uniform upper bound on the variance of all marginals of $X$,
\ref{LCV} requires that the variance of all (conditional) marginals orthogonal to $v$ is uniformly lower bounded by a quantity of the same order.
This clearly helps SVR in not mistaking a generic direction of small variance with the actual direction of the index.
A practical way towards meeting \ref{X} and $\ref{LCV}$ (but also \ref{SMD}) is standardizing the data,
which is the main reason why conditional methods commonly include this preprocessing step.
Since we directly assume \ref{X} and $\ref{LCV}$ (as well as \ref{SMD}),
we did not include standardization as part of SVR,
and therefore of our theoretical analysis.}

Besides the distributional assumptions discussed so far, we introduce the following functional property:
\begin{enumerate}[label = \textnormal{($\Omega$)}]
\item \label{Om}
There are $ \omega \ge 0 $ and \new{$ \L > 0 $} such that, for every interval $T$ with $ |T| \ge \omega $,\\
$ | [ \min f^{-1}(T) , \max f^{-1}(T) ] | \le \new{\L} |T| $.
\end{enumerate}
Assumption \ref{Om} may be regarded as a large scale sub-Lipschitz property.
Note that, if $f$ is bi-Lipschitz, then \ref{Om} is satisfied with $\omega=0$.
However, $\ref{Om}$ for $\omega>0$ does not imply that $f$ is monotone; it relaxes monotonicity to monotonicity ``at scales larger than $\omega$''.

We may now state the main result for the SVR estimator of $v$:

\begin{thm}[SVR] \label{thm:SVR}
Suppose \ref{X},  \ref{Y}, \ref{Z}, \ref{Om}, \ref{SMD} and \ref{LCV} hold true.
Let $\el$ be such that $ 2^{-\el} \lesssim (\L\sqrt{\alpha})^{-1} $ and $ |{\Cijk\el\h}| \gtrsim \max\{\sigma , \omega\} $ for all $ h \in \new{\hh_\el^b} $.
Then, for $n$ large enough so that
$ \tfrac{n}{\sqrt{\tau\log n}} \gtrsim \alpha^2 \L^{2} (t + \el + \log d) d 2^\el $,
we have
 \begin{enumerate}[label=\textnormal{(\alph*)},leftmargin=*]
\item
$
\| \new{P_{\hvc\el^b} - P_v} \| \lesssim \alpha \L \sqrt{t + \el + \log d} \  \sqrt{\frac{ d \ 2^{-\el} }{ n / \sqrt{\tau\log n} }}
$ 
with probability higher than $ 1 - e^{-t} - n^{-\tau} $.
\label{it:SVR-P}
\end{enumerate}
Moreover, if $ \frac{n}{\log n \sqrt{\log n}} \gtrsim \alpha^2 \L^{2} d 2^{\el} $, then
\begin{enumerate}[label=\textnormal{(\alph*)},leftmargin=*,resume]
\item
$
\E [\| \new{P_{\hvc\el^b} - P_v} \|^2] \lesssim \alpha^2 \L^{2} (\el + \log d) \ \frac{d \ 2^{-\el} }{n/\sqrt{\log n}}  \,.
\label{it:SVR-E}
$
\end{enumerate}
If, furthermore, $ | \zeta | \le \sigma $ a.s., then \ref{it:SVR-P} and \ref{it:SVR-E} hold with $n/\sqrt{\log n}$ replaced by $n$ and without $n^{-\tau}$ in \ref{it:SVR-P}.

\end{thm}

Theorem \ref{thm:SVR} not only proves convergence for SVR, but also shows that finer scales give more accurate estimates, provided the number of local samples $n_{\el,\h}^b$ is not too small and we stay above the critical scales $\sigma$ and $\omega$, representing  the noise and the non-monotonicity levels, respectively.
Without assumption \ref{LCM}, both SIR and SVR provide biased estimates of the index vector; it is not known if such bias is removable.
Nevertheless, Theorem \ref{thm:SVR} suggests that the estimation error of SVR could be driven to $0$ by increasing $l$, only limited by the constraint of keeping the scale larger than $\max\{\sigma,\omega\}$.
On the other hand, for distributions not satisfying the assumptions above, the inverse regression curve can deviate considerably from the direction $v$, regardless of the size of the noise (see Figure \ref{fig:SIRvsUS_bias}).
In SVR, assuming for a moment monotonicity ($\omega=0$) and zero noise ($\sigma=0$),
choosing the scale parameter $\el$ according to the lower bound on $n$ yields
a $ O( n^{-2} ) $ convergence rate for the MSE, disregarding logarithmic factors.

To prove Theorem \ref{thm:SVR}, we first establish bounds on the local statistics involved in the computation of the estimator of $v$:

\begin{prop} \label{prop:C}
Suppose \ref{Z} and \ref{Om} hold true.
Let $ T \subset \Ci $ be a bounded interval with $|T|\ge\omega$.
Then:
\begin{enumerate}[label=\textnormal{(\alph*)},leftmargin=*]
\item \label{it:suppC}
\new{For every $  i = 1,\dots,n$ and} every $ \tau \ge 1 $,
$$ \P \{ | \langle v,X_i\rangle - \E[\langle v,X \rangle \mid Y \in T , \new{\bb} ] | \gtrsim \L (|T| + \sqrt{\tau\log n}\ \sigma) \mid Y_i \in T , \new{\bb_i} \} \le 2n^{-\tau}\,. $$
If $ | \zeta | \le \sigma $ a.s., then
$$ \P \{ | \langle v,X_i\rangle - \E[\langle v,X \rangle\mid Y \in T , \new{\bb} ] | \lesssim \L (|T| + \sigma) \mid Y_i \in T , \new{\bb_i} \} =1\,. $$
\item \label{it:varC}
$ \var[ \langle v,X \rangle\mid Y \in T , \new{\bb} ] \lesssim \L^{2} (|T|^2+\sigma^2) $ .
\end{enumerate}
\end{prop}
\begin{proof}
 Let $ Z_t = (-\sqrt{2(t+1)}\sigma,-\sqrt{2t}\sigma] \cup [\sqrt{2t}\sigma,\sqrt{2(t+1)}\sigma) $ for $t\in\N$.
 To prove \ref{it:suppC} we first note that, thanks to \ref{Z},
 we have $ \zeta_i \in \bigcup_{t\le\tau\log n}Z_t $ \new{with probability higher than $ 1- 2n^{-\tau} $, for every $i$}.
Conditioned on this event and on $ Y_i \in T $, $ \langle v,X_i\rangle \in f^{-1}(T+\bigcup_{t\le\tau\log n}Z_t) $.
 On the other hand,
 $ \E [ \langle v,X\rangle\mid Y \in T , \new{\bb} , \zeta \in Z_t ] \in [ \min f^{-1}(T+Z_t) , \max f^{-1}(T+Z_t) ] $.
 It follows from assumption \ref{Om} that
$$
| \langle v,X_i \rangle- \E [ \langle v,X\rangle\mid Y \in T , \new{\bb} , \zeta \in Z_t ] | \lesssim \L ( |T| + \sqrt{\max\{t,\tau\log n\}}\sigma) .
$$
Thus, by the law of total expectation,
\begin{align*}
 | \langle v,X_i\rangle - \E[\langle v,X \rangle\,|\, Y \in T , \new{\bb} ] | & \lesssim \sum_{t=0}^\infty | \langle v,X_i\rangle - \E[ \langle v,X \rangle\,|\, Y \in T , \new{\bb} , \zeta \in Z_t ] | \P\{ \zeta \in Z_t \} \\
 & \lesssim \L \bigg( |T| + \sqrt{\tau\log n}\sigma + \sigma \sum_{t>\tau\log n} \sqrt{t} e^{-t} \bigg) \\
 & \lesssim \L (|T| + \sqrt{\tau\log n}\ \sigma) .
\end{align*}
The case where $ |\zeta| \le \sigma $ almost surely is similar and simpler.
For \ref{it:varC}, we write
\begin{align*}
 \var [ \langle v,X \rangle\mid Y \in T , \new{\bb}] & = \E [ (\langle v,X \rangle- \E[ \langle v,X \rangle\,|\, Y \in T , \new{\bb} ])^2 \,|\, Y \in T , \new{\bb} ] \\
 & = \sum_{t=0}^\infty \E [ (\langle v,X \rangle- \E[ \langle v,X \rangle\,|\, Y \in T , \new{\bb} ])^2 \,|\, Y \in T , \new{\bb} , \zeta \in Z_t ] \P \{ \zeta \in Z_t \} .
\end{align*}
Conditioned on $ \zeta \in Z_t $, assumption \ref{Om} gives
\begin{align*}
 | \langle v,X \rangle- \E[ \langle v,X \rangle\mid Y \in T , \new{\bb} ] | & \lesssim \sum_{s=0}^\infty | \langle v,X \rangle- \E[ \langle v,X \rangle\mid Y \in T , \new{\bb} , \zeta \in Z_s ] | \P\{\zeta\in Z_s \} \\
 & \lesssim \L \bigg( |T| + \sqrt{t} \sigma + \sigma \sum_{s=0}^\infty \sqrt{s} e^{-s} \bigg) \\
 & \lesssim \L (|T| + \sqrt{t}\sigma) ,
\end{align*}
whence
\begin{equation*}
  \var [ \langle v,X \rangle\mid Y \in T , \new{\bb}] \lesssim \L^{2} \bigg( |T|^2 + \sigma^2 \sum_{t=0}^\infty t e^{-t} \bigg) \lesssim \L^{2} ( |T|^2 + \sigma^2) . \qedhere
\end{equation*}
\end{proof}

\begin{prop} \label{prop:eigengap}
Suppose \ref{X}, \ref{Y}, \ref{Z}, \ref{Om}, \ref{SMD} and \ref{LCV} hold true.
Then, for every $l$ such that $ 2^{-\el} \lesssim (\L\sqrt{\alpha})^{-1} $ and $ |{\Cijk\el\h}| \ge \max\{\sigma,\omega\} $, $ h \in \new{\hh_\el^b} $,
$v$ is the eigenvector of smallest eigenvalue of \new{$\Sijk\el\h^b$}, and
$$
 \la_{d-1} ( \new{\Sijk\el\h^b} ) - \la_{d} ( \new{\Sijk\el\h^b} ) \gtrsim R^2 / \alpha .
$$
\end{prop}
\begin{proof}
We have
\begin{align*}
\new{\Sijk\el\h^b} = \Cov [ X \mid Y \in {\Cijk\el\h} , \new{\bb} ] & = \E [ XX^T \mid Y \in {\Cijk\el\h} , \new{\bb} ] - \E [ X \mid Y \in {\Cijk\el\h} , \new{\bb} ] \E [ X^T \mid Y \in {\Cijk\el\h} , \new{\bb} ] .
\end{align*}
Since $X$ is independent of $\zeta$, \eqref{e:SIMdef} implies that $X$ is independent of $Y$ given $\langle v, X\rangle$, hence
\begin{align*}
 \E [ XX^T \mid Y \in {\Cijk\el\h} , \new{\bb} ] & = \E [ \E [ XX^T \mid \langle v,X \rangle, Y \in {\Cijk\el\h} , \new{\bb} ] \mid Y \in {\Cijk\el\h} , \new{\bb} ] \\
 & = \E [ \E [ XX^T \mid \langle v,X \rangle,\new{\bb}] \mid Y \in {\Cijk\el\h} , \new{\bb} ] .
\end{align*}
For the same reason, we have
\begin{align*}
 \E [ X \mid Y \in {\Cijk\el\h} , \new{\bb} ] = \E [ \E [ X \mid \langle v,X \rangle, Y \in {\Cijk\el\h} , \new{\bb} ] \mid Y \in {\Cijk\el\h} , \new{\bb} ]
 = \E [ \E [ X \mid \langle v,X \rangle,\new{\bb}] \mid Y \in {\Cijk\el\h} , \new{\bb} ] .
\end{align*}
Since $X$ satisfies \ref{SMD}, so does $ X \mid \bb $.
Thus, in particular,
$$
 \E [ X \mid \langle v,X \rangle,\bb] = v \ \E [ \langle v,X\rangle \mid \langle v , X \rangle ,
\bb ] .
$$
Now, applying \cite[Theorem 1.a]{10.2307/2669934} to $X\mid\bb$,
we get that $v$ is an eigenvector of \new{$\Sijk\el\h^b$}.
Let $w$ be a unitary vector orthogonal to $v$.
Then
\begin{align*}
 w^T \E [ X \mid Y \in {\Cijk\el\h} , \new{\bb} ] = \new{ \E [ \langle w , v \rangle \E [ \langle v,X \rangle \mid \langle v , X \rangle , \bb ] \mid Y \in {\Cijk\el\h} , \new{\bb} ] } =  0 .
\end{align*}
Moreover, assumption \ref{LCV} gives
\begin{align*}
w^T \E [ XX^T \mid Y \in {\Cijk\el\h} , \new{\bb} ] w = \E [ \var [ \langle w,X\rangle\mid \langle v,X\rangle,\new{\bb}]  \mid Y \in {\Cijk\el\h} , \new{\bb} ] \ge R^2 / \alpha .
\end{align*}
Therefore,
\begin{align*}
 \min_{\substack{w\in\spn\{v\}^\perp \\ \|w\|=1}} w^T \Cov[ X \mid Y \in {\Cijk\el\h} , \new{\bb} ] w \ge R^2 / \alpha .
\end{align*}
To upper bound $ v^T\new{\Sijk\el\h^b}v $, note that $ | {\Cijk\el\h} | = |\Ci| 2^{-\el} \lesssim R 2^{-\el} $.
Thus, assumption \ref{Om} implies by Proposition \ref{prop:C}\ref{it:varC} that
$$
v^T \new{\Sijk\el\h^b} v \lesssim \L^{2} R^2 2^{-2\el} .
$$
We finally put together lower and upper bound.
Taking $ 2^{\el} \gtrsim \L\sqrt{\alpha} $
implies that \new{$ \la_d(\Sijk\el\h^b) $ is the eigenvalue associated to $v$ and}
yields the desired inequality.
\end{proof}

We now establish convergence in probability for the local estimators $ \new{\hvjk\el\h^b} $.

\begin{prop}[local SVR] \label{prop:localSVR}
Suppose \ref{X}, \ref{Y}, \ref{Z}, \ref{Om}, \ref{SMD} and \ref{LCV} hold true.
Then, for every $\el$ such that $ 2^{-\el} \lesssim (\L\sqrt{\alpha})^{-1} $ and $ |{\Cijk\el\h}| \gtrsim \max\{\sigma , \omega\} $ for all $ h \in \new{\hh_\el^b} $, for every $ \eps > 0 $ and $ \tau \ge 1 $,
\begin{align*}
\P \{ \new{ \| P_{\hvjk\el\h^b} - P_v \| } > \eps \mid \new{n_{\el,\h}^b} \}\lesssim d \left[ \exp\!\left({ - \tfrac{c \new{n_{\el,\h}^b} \eps^2}{ \alpha^2 \L^{2} d \sqrt{\tau \log n} (2^{-2\el} + 2^{-\el}\eps)}}\right) \! + \exp\!\left({ - \tfrac{c \new{n_{\el,\h}^b}}{\alpha^2 d}}\right) \right] \! + n^{-\tau} .
\end{align*}
If $ |\zeta| \le \sigma $ a.s., then
\begin{align*}
\P \{ \| \hvjk\el\h - v \| > \eps \mid \new{n_{\el,\h}^b} \} \lesssim d\left[ \exp\left({- \tfrac{c \new{n_{\el,\h}^b} \eps^2}{\alpha^2 \L^{2} d (2^{-2\el} + 2^{-\el}\eps)}}\right) \! + \exp\left({ - \tfrac{c \new{n_{\el,\h}^b}}{\alpha^2 d}}\right) \right] \, .
\end{align*}
\end{prop}
\begin{proof}
Since \new{ $ \| P_{\hvjk\el\h^b} - P_v \| = \sin \angle ( \hvjk\el\h^b , v ) \le 1 $ (see \cite[Ch. 1, Sec. 5.3, Theorem 5.5]{steward1990}) }, we can
assume $ \eps^2 \le \eps \le 1 $ whenever needed.
The Davis--Kahan Theorem  \cite[Theorem VII.3.1]{bhatia}, together with \new{\cite[Ch. 1, Sec. 5.3, Theorem 5.5]{steward1990} and} Proposition \ref{prop:eigengap}, gives
\begin{align*}
 \new{ \| P_{\hvjk\el\h^b} - P_v \| }
\le \frac{ \| v^T (\new{\hSijk\el\h^b} - \new{\Sijk\el\h^b}) \| }{ | \la_{d-1}(\new{\hSijk\el\h^b})  - \la_d(\new{\Sijk\el\h^b})| } \,.
  \end{align*}
  By Proposition \ref{prop:eigengap} and Weyl's inequality \new{\cite[Theorem 4.5.3]{vershynin_2018}} we get
 \begin{align*}
  | \la_{d-1}(\new{\hSijk\el\h^b}) - \la_d(\new{\Sijk\el\h^b}) |
  & \ge \la_d(\new{\Sijk\el\h^b}) - \la_{d-1}(\new{\Sijk\el\h^b}) - | \la_{d-1}(\new{\hSijk\el\h^b}) - \la_{d-1}(\new{\Sijk\el\h^b}) | \\
  & \gtrsim R^2/\alpha - \| \new{\hSijk\el\h^b} - \new{\Sijk\el\h^b} \| .
 \end{align*}
 We bound $ \| \new{\hSijk\el\h^b} - \new{\Sijk\el\h^b} \| $ using the Bernstein inequality \new{\cite[Theorem 5.4.1]{vershynin_2018}}.  
 First, we introduce the intermediate term
\new{
 $$
 \tSijk\el\h^b= \frac{1}{n_{\el,\h}^b} \sum_i (X_i - \mujk\el\h^b) (X_i - \mujk\el\h^b)^T \II (\{Y_i \in \Ci_{\el,\h} \} \cap \bb_i) ,
 $$
 }
 and split $ \new{\hSijk\el\h^b} - \new{\Sijk\el\h^b} $ into
 \new{
 \begin{align*}
  \hSijk\el\h^b - \Sijk\el\h^b = \tSijk\el\h^b - \Sijk\el\h^b - (\hmujk\el\h^b - \mujk\el\h^b) (\hmujk\el\h^b - \mujk\el\h^b)^T .
 \end{align*}
 }
\new{Conditioned on $\bb_i$} we have
 $ \| X_i - \new{\mujk\el\h^b} \|^2 \lesssim R^2 d $, hence
 \new{
$$
  \P \{ \| \tSijk\el\h^b - \Sijk\el\h^b \| \gtrsim R^2/\alpha \mid n_{\el,\h}^b \} \lesssim d \exp \left( -c \frac{n_{\el,\h}^b}{\alpha^2 d } \right) .
 $$
 }
Also, by similar calculations, $ \| \new{\hmujk\el\h^b} - \new{\mujk\el\h^b} \|^2 \lesssim \new{R}^2/\alpha $ with same probability.
We now apply the Bernstein inequality to concentrate $ v^T(\new{\hSijk\el\h^b} - \new{\Sijk\el\h^b}) $.
\new{
We have
\begin{align*}
 \| v^T(\hSijk\el\h^b - \Sijk\el\h^b) \| & \le 
 \| v^T(\tSijk\el\h^b - \Sijk\el\h^b) \| + | v^T(\hmujk\el\h^b - \mujk\el\h^b) | \| \hmujk\el\h^b - \mujk\el\h^b \| \\
 & \lesssim \| v^T(\tSijk\el\h^b - \Sijk\el\h^b) \| + \sqrt{d} R \ | v^T(\hmujk\el\h^b - \mujk\el\h^b) | .
\end{align*}
We start with $ \| v^T(\tSijk\el\h^b - \Sijk\el\h^b) \| $.}
By Proposition \ref{prop:C}\ref{it:suppC}, \new{conditioned on $ Y_i \in \Cijk \el\h $ and $\bb_i$} we have, with probability no lower than $ 1 - 2 n^{-\tau} $,
\begin{align*}
 | v^T (X_i - \new{\mujk\el\h^b}) | \| X_i - \new{\mujk\el\h^b} \| \lesssim \L R^2 \sqrt{d \tau \log n} 2^{-\el} ,
\end{align*}
or $ | v^T (X_i - \new{\mujk\el\h^b}) | \| X_i - \new{\mujk\el\h^b} ] \| \lesssim \L R^2 \sqrt{d} 2^{-\el} $ when $ |\zeta| \le \sigma $.
Next, we estimate the variance.
We have
 \begin{align*}
  \| v^T (\new{\tSijk\el\h^b} - \new{\Sijk\el\h^b}) \|^2 & = v^T (\new{\tSijk\el\h^b} - \new{\Sijk\el\h^b})^2 v \\
  & = v^T (\new{\tSijk\el\h^b})^2 v - v^T \new{\tSijk\el\h^b \Sijk\el\h^b} v - v^T \new{\Sijk\el\h^b \tSijk\el\h^b} v + v^T (\new{\Sijk\el\h^b})^2 v ,
 \end{align*}
hence, taking the expectation,
\begin{align*}
 \E [ \| v^T (\new{\tSijk\el\h^b} - \new{\Sijk\el\h^b}) \|^2 \mid \new{n_{\el,\h}^b} ] = \E [ v^T (\new{\tSijk\el\h^b})^2 v \mid \new{n_{\el,\h}^b} ] - v^T (\new{\Sijk\el\h^b})^2 v,
\end{align*}
where
\begin{align*}
\E [ v^T (\new{\tSijk\el\h^b})^2 v \mid \new{n_{\el,\h}^b} ] & = \frac{1}{(\new{n_{\el,\h}^b})^2} v^T \E \biggl[ \biggl( \sum_i (X_i - \new{\mujk\el\h^b})(X_i - \new{\mujk\el\h^b})^T \biggr)^2 \mid \new{n_{\el,\h}^b} \biggr] v \\
& \le \frac{1}{\new{n_{\el,\h}^b}} v^T \E [ (X\!-\!\new{\mujk\el\h^b}) \|X\!-\!\new{\mujk\el\h^b}\|^2 (X\!-\!\new{\mujk\el\h^b})^T \mid Y\in\Cijk\el\h , \new{\bb} ] v + v^T (\new{\Sijk\el\h^b})^2 v \\
& \le \frac{1}{\new{n_{\el,\h}^b}} dR^2 \E [ (v^T (X-\new{\mujk\el\h^b}))^2 \mid Y\in\Cijk\el\h , \new{\bb}] + v^T (\new{\Sijk\el\h^b})^2 v \\
& = \frac{1}{\new{n_{\el,\h}^b}} dR^2 \var [ v^T X \mid Y\in\Cijk\el\h , \new{\bb} ] + v^T (\new{\Sijk\el\h^b})^2 v .
\end{align*}
Thus, Proposition \ref{prop:C}\ref{it:varC} gives
$$
\E [ \| v^T (\new{\tSijk\el\h^b} - \new{\Sijk\el\h^b}) \|^2 \mid \new{ n_{\el,\h}^b } ] \le \frac{1}{\new{n_{\el,\h}^b}} \L^{2} d R^4 2^{-2\el} .
$$
Therefore, by the Bernstein inequality we obtain
 \begin{equation*}
 \P \{ \| v^T(\new{\tSijk\el\h^b} - \new{\Sijk\el\h^b}) \| > \alpha^{-1} R^2 \eps \mid\new{n_{\el,\h}^b} \} \lesssim d \exp \left( -c \frac{\new{n_{\el,\h}^b} \eps^2}{ \alpha^2 \L^{2} d \sqrt{\tau \log n} \ (2^{-2\el} + 2^{-\el}\eps)} \right) ,
\end{equation*}
without $ \sqrt{\tau \log n} $ if $ | \zeta | \le \sigma $.
\new{We are now left to bound $ v^T(\hmujk\el\h^b-\mujk\el\h^b) $.
By Proposition \ref{prop:C}\ref{it:suppC} we have that,
conditioned on $ Y_i \in \Cijk \el\h $ and $\bb_i$,
with probability higher than $ 1 - 2 n^{-\tau} $,
$$
 | v^T X_i | \lesssim \L R \sqrt{\tau \log n} 2^{-\el} ,
$$
or $ | v^T X_i | \lesssim \L R 2^{-\el} $ if $|\zeta|\le\sigma$.
Moreover, by Proposition \ref{prop:C}\ref{it:varC},
$$  \var [ v^T X \mid Y\in\Cijk\el\h ,\bb] \lesssim L^2 R^2 2^{-2\el} . $$
Thus, the Bernstein inequality gives
$$
 \P \{ \| v^T( \hmujk\el\h^b-\mujk\el\h^b ) \| > \alpha^{-1} d^{-1/2} R \eps \mid\new{n_{\el,\h}^b} \} \lesssim d \exp \left( -c \frac{\new{n_{\el,\h}^b} \eps^2}{ \alpha^2 \L^{2} d \sqrt{\tau \log n} \ (2^{-2\el} + 2^{-\el}\eps)} \right) ,
$$
where again the factor $ \sqrt{\tau \log n} $ can be dropped in the case of bounded noise.
}
 \end{proof}
We are finally in a position to prove Theorem \ref{thm:SVR}.
  \begin{proof}[Proof of Theorem \ref{thm:SVR}]
      \new{
Since $ \la_1(vv^T) = 1 $ and $ \la_2(vv^T) = 0 $,
  }
the Davis--Kahan Theorem \cite[Corollary 1]{SamWan}, \new{along with \cite[Ch. 1, Sec. 5.3, Theorem 5.5]{steward1990}),}  yields
 \begin{align*}
  \new{ \| P_{ \hvc \el^b } - P_v \| }
 & \lesssim
  \new{ ( \la_1(vv^T) - \la_2(vv^T) )^{-1} } \biggl\| \frac{1}{\sum_{\h\in\hh_{\el,\h}^b} \new{n_{\el,\h}^b}} \sum_{\h\in\hh_{\el,\h}^b} \new{\hvjk \el\h^b (\hvjk \el\h^b)^T n_{\el,\h}^b} - vv^T \biggr\| \\
  & = \biggl\| \frac{1}{\sum_{\h\in\hh_{\el,\h}^b} \new{n_{\el,\h}^b}} \sum_{\h\in\hh_{\el,\h}^b} (P_{\hvjk \el\h^b} - P_v) n_{\el,\h}^b \biggr\| \\
 & \le \frac{1}{\sum_{\h \in \hh_{\el,\h}^b} n_{\el,\h}^b} \sum_{h\in\hh_{\el,\h}^b} \| P_{\hvjk \el\h^b} - P_v \| n_{\el,\h}^b .
\end{align*}
\new{
For each $ h \in \hh_\el $ we have $ n_{\el,\h}^b \ge 2^{-\el} n^b $,
where $ n^b \gtrsim n $ with probability higher than $ 1 - 2 e^{-cn} $ by Lemma \ref{lem:sub-gauss-ball},
hence $ n_{\el,\h}^b \gtrsim 2^{-\el} n $.
We thus apply Proposition \ref{prop:localSVR}.
Taking the union bound over $ \h \in \hh_\el $ gives
$ \| P_{\hvjk\el\h^b} - P_v \| \le \eps $ for all $ h \in \hh_\el $ with probability higher than
 \begin{align*}
1 - C 2^{\el} d \left[ \exp\!\left({ - \tfrac{c 2^{-\el} n \eps^2}{ \alpha^2 \L^{2} d \sqrt{\tau \log n} (2^{-2\el} + 2^{-\el}\eps)}}\right) \! + \exp\!\left({ - \tfrac{c 2^{-\el} n}{\alpha^2 d}}\right) \right] \! - n^{-\tau} .
\end{align*}
We now constrain $ \eps \le 2^{-\el} $. so that the two exponentials are bounded by
\begin{align} \label{eq:int0}
 \exp\!\left({ - \tfrac{c 2^{-\el} n \eps^2}{ \alpha^2 \L^{2} d \sqrt{\tau \log n} (2^{-2\el} + 2^{-\el}\eps)}}\right) \! + \exp\!\left({ - \tfrac{c 2^{-\el} n }{\alpha^2 d}}\right)
 & \le  \exp\!\left({ - \tfrac{c 2^{-\el} n \eps^2}{ \alpha^2 \L^{2} d \sqrt{\tau \log n} 2^{-2\el}}}\right) \! + \exp\!\left({ - \tfrac{c 2^{-\el} n \eps^2}{\alpha^2 d 2^{-2\el}}}\right) \nonumber \\
 & \lesssim \exp\!\left({ - \tfrac{c n \eps^2}{ \alpha^2 \L^{2} d \sqrt{\tau \log n} 2^{-\el}}}\right) . \tag{$*$}
\end{align}
Then, for $ t > 0 $, we take $ \eps = c^{-1/2} \alpha \L \sqrt{t + \el + \log d} \ 2^{-\el/2} \sqrt{\frac{ d }{ n / \sqrt{\tau\log n} }} $,
with the constraint $ \eps \le 2^{-\el} $ translating to
$ \tfrac{n}{\sqrt{\tau\log n}} \gtrsim \alpha^2 \L^{2} (t + \el + \log d) d 2^\el $,
which leads to \ref{it:SVR-P}.
}
For \ref{it:SVR-E}, we first condition on $ \new{\mathcal{Z}} = \{ |\zeta_i| \le \sqrt{2\tau\log n} \sigma \text{ for all $i$'s} \} $ and calculate
\begin{align} \label{eq:int1}
 \E [ \| \new{P_{\hvc\el^b} - P_v} \|^2 ] - n^{-\tau} & \lesssim \int_0^1 \eps \ \P \{ \| \new{P_{\hvc\el^b} - P_v} \| > \eps \mid \new{\mathcal{Z}} \} d\eps \nonumber \\
 & = \int_0^{2^{-\el}} \eps \ \P \{ \| \new{P_{\hvc\el^b} - P_v} \| > \eps \mid \new{\mathcal{Z}} \} d\eps + \int_{2^{-\el}}^1 \eps \ \P \{ \| \new{P_{\hvc\el^b} - P_v} \| > \eps \mid \new{\mathcal{Z}} \} d\eps \nonumber \\
 & \le \int_0^{2^{-\el}} \min \left \{ 1 , 2^\el d \exp \left( -\tfrac{c (n/\sqrt{\tau\log n}) \eps^2}{\alpha^2 \L^{2} d 2^{-\el}} \right) \right\} \eps d\eps \nonumber \\
 & + \int_{2^{-\el}}^1 2^\el d \exp \left( - \tfrac{c (n/\sqrt{\tau\log n})}{\alpha^2 \L^{2} d 2^{\el}} \right) \eps d\eps \nonumber \\
 & \lesssim \alpha^2 \L^{2} d \log(2^\el d) \frac{2^{-\el}}{n/\sqrt{\tau\log n}}
 + 2^\el d \exp \left( -\tfrac{c(n/\sqrt{\tau\log n} \new{)} }{\alpha^2 \L^{2} d 2^{\el}} \right) , \tag{$**$}
\end{align}
where
\new{
for the integral $ \int_0^{2^{-\el}} $ we have used \eqref{eq:int0},
for the integral $ \int_{2^{-\el}}^1 $ we have bounded
\begin{align*}
 \exp\!\left({ - \tfrac{c 2^{-\el} n \eps^2}{ \alpha^2 \L^{2} d \sqrt{\tau \log n} (2^{-2\el} + 2^{-\el}\eps)}}\right)
& \le \exp\!\left({ - \tfrac{c 2^{-\el} n \eps^2}{ 2 \alpha^2 \L^{2} d \sqrt{\tau \log n} 2^{-\el} \eps }}\right) \\
& = \exp\!\left({ - \tfrac{c n \eps}{ 2 \alpha^2 \L^{2} d \sqrt{\tau \log n} }}\right) \\
& \le \exp\!\left({ - \tfrac{c n 2^{-\el}}{ 2 \alpha^2 \L^{2} d \sqrt{\tau \log n} }}\right) ,
\end{align*}
}
and the last inequality follows from Lemma \ref{lem:P-E}.
For $\tau=2$ and $ n $ large enough \new{so that $ \tfrac{n}{\sqrt{\log n}} \gtrsim \alpha^2 \L^{2} d 2^{-\el} \log n $,
the second term in \eqref{eq:int1} is bounded by the first one}, and we obtain \ref{it:SVR-E}.
Analogous computations for the case where $ |\zeta| \le \sigma $ lead to the final claim.
\end{proof}

\subsection{Conditional regression bounds} \label{sec:regression}

In this section we study how partitioning polynomial regression of the link function in a single-index model is affected by an estimate $\hv$ of the index vector,
where regression estimators are viewed as conditioned on $\hv$.
\new{
For simplicity, we assume that $\hv$ was trained on a separate sample of $ (X,Y) $
of cardinality $n$.
As a consequence, $\hv$ is independent of the samples $ \{ (X_i,Y_i) \}_{i = 1}^n $ used to regress the link function.
If we are given a single training set,
independence can be easily enforced by randomly splitting the data in half.
In this case, all rates are understood to scale with $n/2$,
where the factor $1/2$ is absorbed by the other absolute constants.
}

We will focus on one standard class of priors for regression functions,
namely the class $\cc^s$ of H\"older continuous functions.
We recall that a function $ g : \R^d \to \R $ is $ \cc^s$ H\"older continuous ($ g \in \cc^s $) if,
for $ s = m + \alpha $, $ m \ge 0 $ an integer and $ \alpha \in (0,1] $,
$g$ has bounded continuous derivatives up to order $m$ and
$$
 |g|_{\cc^s} = \max_{|\la|=m} \sup_{x\ne z} \frac{\partial^\la g(x) - \partial^\la g(z)}{\|x-z\|^\alpha} < \infty .
$$
The H\"older norm is defined by
$$
 \| g \|_{\cc^s} = \sum_{|\la| \le m} \| \partial^\la g \|_\infty + | g |_{\cc^s} .
$$
\new{In expectation, piecewise polynomial estimators of degree $m$ are min-max optimal on the class $ \cc^{s} $ for all $ s = m + \alpha $ with $ m \ge 0 $ and $ \alpha \in (0,1] $ \cite[Corollary 11.2]{Gyorfi}.
Our analysis will be limited to polynomials of degree $ m \in \{ 0 , 1 \} $,
and thus to functions of smoothness $ s \le 2 $}.
On the other hand, piecewise polynomials of order greater than zero are in general not optimal in high probability \cite[Section 3]{BCDD2}.
For this reason, in the case of a $\cc^s$ H\"older regression function with $s\in(1,2]$
we will assume the following regularity condition:
\begin{enumerate}[label=\textnormal{(R)}]
\item \label{R} \new{Let $\rho$ be the distribution of $X$, and let $ \rho_v $ denote the push-forward measure of $\rho$ along the map $ x \in \R^d \mapsto \langle x , v \rangle \in \R $.
For every interval $ I \subset \supp \rho_v $,
$ \var [ \langle v , X \rangle \mid \langle v , X \rangle \in I ] \gtrsim |I|^2 $.
}
\end{enumerate}
To control the distributional mismatch of the projection $ \langle \hv , X \rangle $,
we will also make one of the following two assumptions:
\begin{enumerate}[label=\textnormal{(P1)}]
\item \label{P1} $X$ has spherical distribution.
\end{enumerate}
\begin{enumerate}[label=\textnormal{(P2)}]
\item \label{P2} \new{Let $\rho$ be the distribution of $X$, and let $ \rho_v $ denote the push-forward measure of $\rho$ along the map $ x \in \R^d \mapsto \langle x , v \rangle \in \R $.
The distribution $\rho$ has upper bounded density with bounded support,
and for every interval $ I \subset \supp \rho_v $,
$ \rho_v( I ) \gtrsim |I| $}.
\end{enumerate}
As discussed early in Section \ref{s:AnalysisConvergence},
spherical distributions \ref{P1} provide a customary model for conditional regression estimators
and cover, but are not limited to, the Gaussian distribution.
On the other hand, the class \ref{P2} includes a variety of regular densities on compact normal domains, where no special symmetry is required.

We now state our main theorem:
\begin{thm} \label{thm:main}
Assume \ref{X} and \ref{Z}.
Furthermore, assume either \ref{P1} or \ref{P2}.
Let $\hv$ be an estimator for $v$ such that, for every $ \eps > 0 $,
\begin{enumerate}[label = \textnormal{($\widehat{\text{V}}$)}]
\item \label{J} $ \P \{ \| \new{P_\hv} - \new{P_v} \| > \eps \} \le A \exp( - n \eps^2 / B) $ 
\end{enumerate}
for some $ A , B \ge 1 $ possibly dependent on $d$ and specific parameters.
Suppose $ f \in \cc^s $ with $ s \in [1/2,2] $,
and assume \ref{R} in the case $ s>1 $.
Let $ \widehat{F}_{j|\hv} $ be a piecewise constant $(s\le1)$ or linear $(s>1)$ estimator of $F$ at scale $j$ conditioned on $\hv$
\new{and truncated at $ M = \|F\|_\infty $},
as defined in Section \ref{sec:hf} for $ \langle \hv , x \rangle \in I = [-\ar,\ar] $, and $0$ outside.
Then, setting $ 2^{-j} \asymp \sqrt{B} (\log n / n)^{1/(2s+1)} $ and $ r \asymp \sqrt{2d\log n} R $ we have:
 \begin{enumerate}[label=\textnormal{(\alph*)},leftmargin=*]
 \item \label{it:main-prob} For every $ \nu > 0 $ there is $ c_\nu(d,B,R,\|f\|_{\cc^s},s) \ge 1 $ such that \\
 $
  \P \left\{ ( \E_X [ | \hFjhv(X)- F(X) |^2 ] )^{\frac1 2} > \new{c_\nu} (\log n)^{s/2} \left( \frac{\log n}{n} \right)^{\frac{s}{2s+1}} \right\} \lesssim A n^{-\nu}
 $.
 \Item \label{it:main-exp}
 $
 \E [ | \hFjhv(X) - F(X) |^2 ] \le K (\log n)^s \left( \frac{\log n}{n} \right)^{\frac{2s}{2s+1}}
 $ \\

\vspace{\baselineskip}
 for some $ K = K(d,A,B,R,\|f\|_{\cc^{s}},s) $.
 \end{enumerate}
The dependence of all constants upon $d$, $A$ and $B$ is polynomial.
 \end{thm}
 
 Theorem \ref{thm:main} shows that partitioning poliynomial estimators \new{(of degree $0$ and $1$)} achieve the $1$-dimensional min-max convergence rate (up to logarithmic factors) \new{(for $s$-H\"older functions with $s\in[1/2,2]$)} when conditioned on \emph{any} $\sqrt{n}$-convergent estimate of $v$,
and thus, \emph{in particular}, on the estimate $\widehat v$ obtained with SVR
(under its assumptions).
The same conclusion follows for other prominent index estimators,
including conditional methods such as SIR, SAVE, SCR and DR.
Although formally our theorem requires non-asymptotic $\sqrt{n}$-convergence to the index,
and only asymptotic $\sqrt{n}$-consistency has been established for the aforementioned methods,
finite-sample bounds can be derived as well with similar arguments to the ones employed in Section \ref{sec:index-bounds}.

\new{
The reader will note that the parameter choices made in Theorem \ref{thm:main}
depend on quantities which might be unknown,
notably the bound $ \|F\|_\infty $ in the truncation level $M$
and the smoothness $s$ in the polynomial degree $ m $ and scale $ 2^{-j} $.
While such results are customary in statistical learning theory,
as they still provides the existence of an optimal choice for the parameters,
and thus of a suitable range for cross-validation,
we briefly discuss here this limitation.
The choice $ M = \| F \|_\infty $ is made mostly for simplicity.
In many applications, at least a proxy $ M \ge \|F\|_\infty $ is known.
For tight proxies $ \|F\|_\infty \le M \lesssim \|F\|_\infty $,
the bounds in Theorem \ref{thm:main} hold true up to an additional constant factor.
Otherwise, the constants $c_\nu$ and $K$ would also depend on $M$.
For what concerns the knowledge of $s$, one could use again an upper bound.
An algorithmic solution would instead be trying polynomials of increasing degree
and cross-validate the scale.
Smoothness-independent theoretical bounds are possible
choosing the partition adaptively by wavelet-thresholding techniques \cite{BCDDT1,BCDD2,GMRARegression}.
}

To prove Theorem \ref{thm:main} we first show that, with high probability,
a conditional polynomial estimator \new{(of degree $0$ or $1$)} differs from an oracle estimator (possessing knowledge of $v$)
by the angle between $\widehat v$ and $v$:
\begin{prop} \label{prop:hFhv-hFv}
Assume \ref{X} and \ref{Z}.
Furthermore, assume either \ref{P1} or \ref{P2}.
Suppose $ f \in \cc^\alpha $ with $ \alpha \in [1/2,1] $.
Let $\hv$ be an estimate of $v$ \new{with $ \langle \hv , v \rangle \ge 0 $}.
For $ u \in \{ v , \hv \} $,
let $ \widehat{F}_{j|u} $ be a piecewise constant ($ \alpha \le 1 $) or linear ($ \alpha = 1 $) estimator of $F$ at scale $j$ conditioned on $u$
\new{and truncated at $ M = \|F\|_\infty $},
as defined in Section \ref{sec:hf} for $ \langle u , x \rangle \in I = [-\ar,\ar] $, and $0$ outside.
Assume \ref{R} in case of a linear estimator.
Then, for every $ \eps > 0 $, $\ar \ge 1 $ and conditioned on $ 2^{-j} \ge \| \hv - v \|/t $ for some $ t \ge 1 $,
we have
\begin{align*}
\big (\E_X[ | \hfjhv(\langle \hv , X \rangle) - \hfjv(\langle \hv , X \rangle) |^2 \II \{ X \in B(0,\ar) \} ]\big)^{\frac1 2}
\lesssim t |f|_{\cc^\alpha} \ar^{\frac{1}{2-\alpha}} \| \hv - v \|^{\frac{1}{2-\alpha}} + \eps
\end{align*}
 with probability higher than $ 1 - C \#\K j \exp\left(-\frac{c \ n \eps^2}{\#\K j t^2 \|f\|_{\cc^\alpha}^2 \ar^{2\alpha}} \right) - 2n \exp(-\ar^2/2dR^2) $~.
\end{prop}
The proof of Proposition \ref{prop:hFhv-hFv} can be found in Appendix \ref{sec:proofsmain}.
The key tool to obtain the dependence on $\|\hv-v\|$ in the upper bound is the Wasserstein distance,
which enables to control the difference between statistics computed on the conditional distribution given $\hv$ rather than $v$.
We now proceed to prove Theorem \ref{thm:main}.
\begin{proof}[Proof of Theorem \ref{thm:main}]
\new{
First, we observe that, since $ F(x) = f ( \langle v , x \rangle ) $,
defining $ \check{f}(t) = f(-t) $
we have $ F(x) = \check{f} ( \langle - v , x \rangle ) $ as well.
Moreover, $ f \in \cc^s $ if and only if $ \check{f} \in \cc^s $,
with $ |f|_{\cc^{t}} = |\check{f}|_{\cc^{t}} $ for all $ t \le s $.
Thus, since we can always replace $v$ with $-v$ and $f$ with $\check{f}$,
in the following we may assume without loss of generality that
$ \langle \hv , v \rangle \ge 0 $,
hence $ \| \hv - v \| \le \sqrt{2} \| P_\hv - P_v \| $.
}

Let $ m = \lceil s \rceil - 1 \in \{ 0 , 1 \} $ and $ \alpha=s\wedge1 \in [\tfrac{1}{2} , 1 ] $.
We start by isolating the error outside a ball $ B(0,\ar) $:
\begin{align*}
 & \E_X [ | F(X) - \hFjhv(X) |^2 ]
 \lesssim \E_X [ | F(X) - \hFjhv(X) |^2 \II \{ \|X\| \le \ar \} ]
 + |f|_{\cc^0}^2  \P \{ \|X\| > r \} ,
 \end{align*}
 whence, using Lemma \ref{lem:sub-gauss-tail}, we get the tail bound
\begin{align} \label{eq:tail}
 \P \{ \|X\| > r \} \lesssim \exp(-\ar^2/2dR^2) \tag{T} .
\end{align}
For $ x \in B(0,r) $ we decompose
\begin{align*}
 | F(x) - \hFjhv(x) | & = | f ( \langle v , x \rangle ) - \hfjhv( \langle \hv , x \rangle ) |
\le \underbrace{| f ( \langle v , x \rangle ) - f ( \langle \hv , x \rangle ) |}_{\text{(\straighttheta1)}}
 + | f ( \langle \hv , x \rangle ) - \hfjhv( \langle \hv , x \rangle ) | ,
\end{align*}
and bound (\straighttheta1) by the angle $ \| \hv - v \| $:
\begin{align*}
 \E_X [ | f ( \langle v , X \rangle ) - f ( \langle \hv , X \rangle ) |^2 \II \{ \|X\| \le \ar \} ]
 \le |f|_{\cc^\alpha}^2 r^{2\alpha} \| \hv - v \|^{2\alpha} .
\end{align*}
Hence, from assumption \ref{J} and Lemma \ref{lem:P-E} we get
\begin{align}
& \P \{ \E_X [ | f ( \langle v , X \rangle ) - f ( \langle \hv , X \rangle ) |^2 \II \{ \|X\| \le \ar \} ] > \eps^2 \} \le A \exp\left( - c \tfrac{ n \eps^{2/\alpha} }{ B \ar^2 |f|_{\cc^\alpha}^{2/\alpha} } \right) \label{eq:angle1a} \tag{$\Theta1$a} \\
& \E [ | f ( \langle v , X \rangle ) - f ( \langle \hv , X \rangle ) |^2 \II \{ \|X\| \le \ar \} ]
\le  (\log(A)B)^\alpha |f|_{\cc^\alpha}^2 \ar^{2\alpha}n^{-\alpha} \label{eq:angle1} \tag{$\Theta1$b} .
\end{align}
We further decompose
\begin{align*}
  & | f ( \langle \hv , x \rangle ) - \hfjhv( \langle \hv , x \rangle ) | \\
  \le \ & \underbrace{| f ( \langle \hv , x \rangle ) - \fjv( \langle \hv , x \rangle ) |}_{\text{(b)}}
  + \underbrace{| \fjv( \langle \hv , x \rangle ) - \hfjv( \langle \hv , x \rangle ) |}_{\text{(v)}}
  + \underbrace{| \hfjv( \langle \hv , x \rangle ) - \hfjhv( \langle \hv , x \rangle ) |}_{\text{(\straighttheta)}} ,
\end{align*}
\new{where $\fjv$ is the $m$-order population estimator of $f$ at scale $j$ conditioned on $v$, namely
\begin{align*}
 & f_{j,k|v} = \argmin_{\deg(p) \le m} \E [ | Y - p(\langle v , X \rangle) |^2 \II\{ \langle v , X \rangle \in I_{j,k} \} ] , \\
 & \fjv(t) = \sum_{k\in\kk_j} T_M[ f_{j,k|v}(t) ] \II\{t\in I_{j,k}\} .
\end{align*}}
Integrating (b) we get a bias term,
that we control exploiting the H\"older continuity of $f$ (\new{see \cite[Appendix A, Example 1]{GMRARegression}\footnote{\label{note1}\new{
\cite{GMRARegression} deals more generally with piecewise polynomial regression
on an unknown manifold $\mm$ of dimension $d$.
In our setting,
$\mm$ is a known Euclidean domain of dimension $d=1$, in fact an interval.
In particular, the dyadic decomposition \ref{it:Ijk} satisfies assumptions (A1)$\div$(A4)
in \cite{GMRARegression} with $\theta_1=1$ and $\theta_2=2\ar$.
Assumption (A5)(i) is satisfied thanks to \ref{R},
while (A5)(ii) is trivial.
}
})}:
\begin{align} \label{eq:bias}
 \E_X [ | f ( \langle \hv , X \rangle ) - \fjv( \langle \hv , X \rangle ) |^2 \II \{ \|X\| \le \ar \} ] \lesssim |f|_{\cc^s}^2 \ar^{2s} 2^{-2js} \tag{B} .
\end{align}
The variable (v) leads to a variance term,
which can be concentrated with known calculations (\new{see \cite[Proposition 2 and Lemma 5]{GMRARegression}\footref{note1}}):
\begin{align}
& \P \{ \E_X [ | \fjv( \langle \hv , X \rangle ) \! - \! \hfjv( \langle \hv , X \rangle ) |^2 \II \{ \|X\| \le \ar \} ] > \eps^2 \} \lesssim \#\K j \exp \left( -c \tfrac{n\eps^2}{\#\K j |f|_{\cc^0}^2 } \right) \label{eq:variancea} \tag{Va} \\
 & \E [ |  \fjv( \langle \hv , X \rangle ) - \hfjv( \langle \hv , X \rangle ) |^2 \II \{ \|X\| \le \ar \} ]
\lesssim |f|_{\cc^0}^2 \frac{\log( \#\K j)  \#\K j }{n} . \label{eq:variance} \tag{Vb}
\end{align}
For (\straighttheta) we condition on the event $ \| \hv - v \| \le t 2^{-j} $,
taking into account that, thanks to assumption \ref{J}, the complement has probability
\begin{align} \label{eq:angle2}
 \P \{ \| \hv - v \| > \new{t}2^{-j} \} \le A \exp(- n t^2 2^{-2j} / B) \tag{$\Theta2$} .
\end{align}
Thus, Proposition \ref{prop:hFhv-hFv} along with assumption \ref{J} and Lemma \ref{lem:P-E} gives
\begin{align}
& \P \{ \E_X [ | | \hfjv( \langle \hv , X \rangle ) - \hfjhv( \langle \hv , X \rangle ) |^2 \II \{ \|X\| \le \ar \} ] > \eps^2 \mid \| \hv - v \| \le t 2^{-j} \} \nonumber \\
\lesssim \ & \#\K j \exp\left(-c \tfrac{n \eps^2}{\#\K j t^2 |f|_{\cc^\alpha}^2 \ar^{2\alpha}} \right) + A \exp\!\left( \! - \tfrac{n \eps^{2(2-\alpha)}}{B t^{2(2-\alpha)} |f|_{\cc^\alpha}^{2(2-\alpha)} \ar^{2} } \!\right) , \label{eq:angle3a} \tag{$\Theta3$a} \\
& \E [ | \hfjv(\langle \hv , X \rangle) - \hfjhv(\langle \hv , X \rangle) |^2 \II \{ \|X\| \le \ar \} \mid \| \hv - v \| \le 2^{-j} ] \nonumber \\
 \lesssim \ &  |f|_{\cc^\alpha}^2 \ar^{2\alpha} \frac{\log(\#\K j) \#\K j}{n}
 + (\log(A) B)^{\frac{1}{2-\alpha}} |f|_{\cc^\alpha}^2 \ar^{\frac{2}{2-\alpha}} n^{-\frac{1}{2-\alpha}} . \label{eq:angle3} \tag{$\Theta3$b}
\end{align}

In order to balance the tail \eqref{eq:tail}, the bias \eqref{eq:bias}, the variance \eqref{eq:variance} and the angle terms \eqref{eq:angle1}, \eqref{eq:angle2} and \eqref{eq:angle3},
we choose
$$
r = \sqrt{2d\log n} R\,, \qquad 2^{-j} \asymp \sqrt{B} (\log n / n)^{1/(2s+1)} , \qquad t = 1 ,
$$
and plug in $ \#\K j = 2^j $, which leads to
\begin{align*}
\E [ | \hFjhv(X) - F(X) |^2 ] & \lesssim |f|_{\cc^0}^2 \left(\frac{1}{n}\right)  & \eqref{eq:tail} \\
& + |f|_{\cc^\alpha}^2 R^{2\alpha} (\log(A)B d)^\alpha \left(\frac{\log n}{n}\right)^{\alpha} & \eqref{eq:angle1} \\
& + |f|_{\cc^s}^2 R^{2s} (B d)^s (\log n)^s \left(\frac{\log n}{n}\right)^{\frac{2s}{2s+1}} & \eqref{eq:bias} \\
& + |f|_{\cc^0}^2 \tfrac{1}{2s+1} \left(\frac{\log n}{n}\right)^{\frac{2s}{2s+1}} & \eqref{eq:variance} \\
& + A \exp\left( - n \left(\tfrac{\log n}{n} \right)^{\frac{2}{2s+1}} \right) & \eqref{eq:angle2} \\
& + |f|_{\cc^\alpha}^2 \tfrac{1}{2s+1} R^{2\alpha} d^\alpha (\log n)^\alpha \left(\frac{\log n}{n}\right)^{\frac{2s}{2s+1}} & \eqref{eq:angle3} \\
& + |f|_{\cc^\alpha}^2 R^{\frac{2}{2-\alpha}} (\log(A) B)^{\frac{1}{2-\alpha}} d^{\frac{1}{2-\alpha}} \left(\frac{\log n}{n}\right)^{\frac{1}{2-\alpha}} .
\end{align*}
In \eqref{eq:angle1}, $ (\log n / n)^{\alpha} \le (\log n / n)^{2s/2s+1} $
for $ \alpha = s \in [1/2,1] $, and for $ \alpha = 1 $ and $ s > 1 $.
In \eqref{eq:angle2}, $ \exp( - n (\log / n)^{2/(2s+1)} ) \le 1/n $ for $ s \ge 1/2 $.
In \eqref{eq:angle3}, $ ( \log n / n )^{1/(2-\alpha)} \le ( \log n / n )^{2s/(2s+1)} $
for $ \alpha = s \in (0,1] $, and for $ \alpha = 1 $ and $ s > 1 $.
Collecting the constants we obtain \ref{it:main-exp}.
\new{The bound \ref{it:main-prob} follows similarly from \eqref{eq:tail}, \eqref{eq:bias}, \eqref{eq:variancea}, \eqref{eq:angle1a}, \eqref{eq:angle2} and \eqref{eq:angle3a}} setting
$$
 \new{r = \sqrt{2d\log n} R\,, \quad 2^{-j} \asymp \sqrt{B} (\log n / n)^{1/(2s+1)} ,} \quad \eps = c_\nu \log^{\frac{s}{2}} n \ (\log n / n)^{s/(2s+1)} ,
 \quad t = \sqrt{c_\nu} ,
$$
and taking $c_\nu$ large enough.
\end{proof}

The additional logarithmic factors in \ref{it:main-prob} and \ref{it:main-exp}
are exclusively due to the unboundedness of the distribution
and can be avoided in the bounded case.
While here we are restricting to constant and linear estimators, hence to the smoothness range $ s \le 2 $, the results may be extended to higher order polynomials, and thus to smoother regression functions, with similar proofs.
However, from our decomposition, and specifically from \eqref{eq:angle1},
it seems not possible to maintain min-max optimality in the range $ s < 1/2 $.
We remark that this smoothness constraint does not depend on the regression technique,
that is, the term \eqref{eq:angle1} would still arise when using different regression methods than partitioning polynomials.

\section{Numerical experiments}
\label{s:NumericalExperiments}

In this section we conduct numerical experiments to demonstrate that the theoretical results above have practical relevance and to investigate how relaxations of the assumptions affect the estimators. In order to highlight specific aspects of different algorithms we use three different functions to conduct our experiments.
The first two are
\begin{equation*}
F_1(x)=\exp(\langle v,x \rangle/3)) \,,\qquad F_2(x)= F_1(x) + \sin(20 \langle v,x\rangle)/15   \,.
\end{equation*}
Both  functions are smooth. $F_1$ is monotone and thus we may choose $\omega=0$, while $F_2$ is non-monotone, thus condition \ref{Om} is satisfied only for $\omega>0$.
This allows us to explore the behavior of $\widehat v$ under monotonicity or lack thereof, and how the estimators are effected by the choice of the scales $l$ and $j$. 
To investigate the convergence rate of the regression estimator $\widehat F$, we use a monotone function $F_3$ which is piecewise quadratic on a random partition and continuous. The domain of $x$, and its dimension $d$, will be specified in each experiment.
$F_1, F_2, F_3$ are shown in Figure \ref{fig:different_functions}.

\begin{figure}[h]
\includegraphics[width= 0.6 \textwidth]{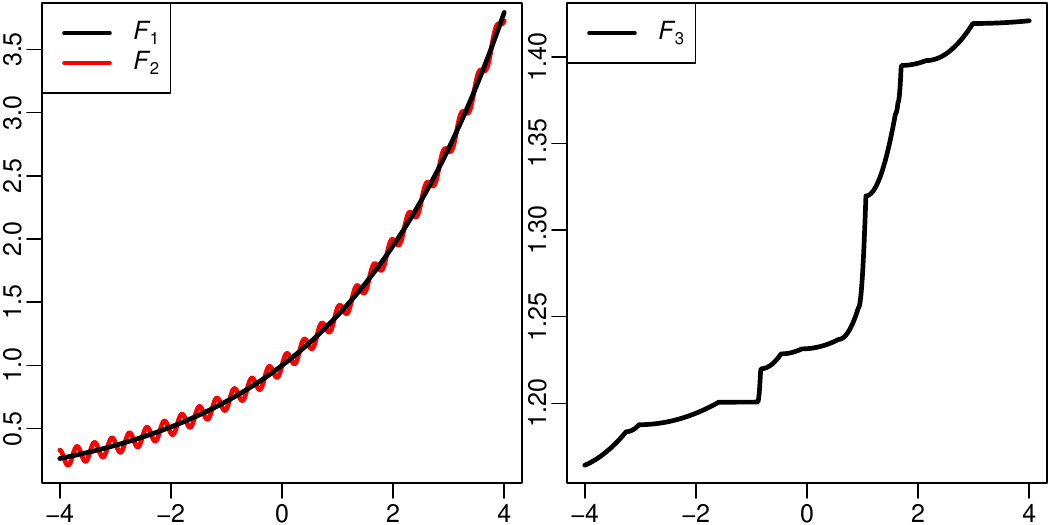}
\caption{Different functions used in the experiments, with horizontal axis representing $\langle v,x\rangle$.}
\label{fig:different_functions}
\end{figure}

\subsection{Estimating the index vector}
\label{s:numestv}
Here we compare the performances of SIR, SAVE and SVR in estimating the index vector $v$.
We consider two settings S1, S2, corresponding to two different non-elliptical, and thus non-spherical, distributions for $X$: $\rho_{X,1},\rho_{X,2}$; $\rho_{X,1}$ is a standard normal $\mathcal{N}(0,1)$ in one coordinate, and a skewed normal with shape parameter $\alpha=5$ in the other coordinate;  $\rho_{X,2}$ is uniform on the triangle with vertices (0,0),(1,1),(0,1); all distributions are normalized to have zero mean and standard deviation equal to one. Note that in both settings conditions \ref{LCM} and \ref{SMD} are not satisfied.
For each setting we draw $n=1000$ i.i.d. samples and generate the response variable $Y_i=F(X_i)+\zeta_i$ using functions $F_1$ and $F_2$, where $\zeta_i\sim\mathcal{N}(0,\sigma^2)$. We use different levels of noise setting $\sigma$ equal to the $0\%$, $1\%$ and $2\%$ of $|f(-4)-f(4)|$.
We chose $ v= (1/\sqrt{5},2/\sqrt{5})$ for setting S1, and $v=(1,0)$ for setting S2.
The results in Table \ref{tab:grad_err_and_time} show the detailed performance of SIR, SAVE and SVR for all settings, functions, and noise levels.
First, we note that the cases of $F_1$ with $1\%$ noise and $F_2$ with zero noise produce similar results.
This is consistent with the intuition manifested from Theorem \ref{thm:SVR} that noise and non-monotonicity levels play a similar role in the accuracy of the estimators.
In all settings SIR performs worse than the two other methods, SVR and SAVE, which on the other hand have similar performance, although SVR produces most of the times slightly better estimates.  
\begin{table}[h]
\caption{Performance of the different algorithms in different settings,
with $\mathrm{err}=\log_{10}(\| \widehat{v} - v \|^2)$, corresponding standard error, and average computational time in $\textnormal{seconds}/100$.  \label{tab:grad_err_and_time}
}
\ra{1.2}
\begin{tabular}{c c l r r r c r r r}
\multicolumn{3}{c}{} &  \multicolumn{3}{c}{\textbf{S1}} & &  \multicolumn{3}{c}{\textbf{S2}} \\
  \cline{4-6} \cline{8-10}
   & $\pmb{\sigma}$ &  & \multicolumn{1}{c}{\textbf{err}} & \multicolumn{1}{c}{\textbf{se}} & \multicolumn{1}{c}{\textbf{time}} & & \multicolumn{1}{c}{\textbf{err}} & \multicolumn{1}{c}{\textbf{se}} & \multicolumn{1}{c}{\textbf{time}} \\
  \cline{2-10}
   \multirow{3}{*}{$\pmb F_1$}
 & \multirow{6}{*}{$0\%$}
   & \textbf{SIR} & -3.04 & -2.83 & 0.60 & & -0.68 & -1.83 & 0.60 \\ 
    & & \textbf{SAVE} & -5.97 & -1.06 & 1.30 & & -7.41 & -7.14 & 1.40 \\ 
  & &  \textbf{SVR} & -6.42 & -6.06 & 1.00 & & -7.60 & -6.65 & 0.70 \\ 
  \cline{3-10}
 \multirow{3}{*}{$\pmb F_2$}  
 &
   &  \textbf{SIR} & -3.11 & -2.99 & 0.50 & & -0.67 & -1.83 & 0.50 \\ 
    & &   \textbf{SAVE} & -4.39 & -4.10 & 1.30 & & -4.13 & -4.01 & 1.30 \\ 
 & &   \textbf{SVR} & -4.41 & -4.08 & 0.90 & & -4.16 & -3.92 & 0.70 \\ 
  \cline{2-10}
 \multirow{3}{*}{$\pmb F_1$}  
 & \multirow{6}{*}{$1\%$}
   &   \textbf{SIR} & -2.97 & -2.69 & 0.50 & & -0.68 & -1.81 & 0.50 \\ 
    & &  \textbf{SAVE} & -4.57 & -4.24 & 1.40 & & -4.16 & -4.03 & 1.30 \\ 
 & &   \textbf{SVR} & -4.58 & -4.28 & 1.00 & & -4.12 & -3.75 & 0.80 \\ 
  \cline{3-10}
 \multirow{3}{*}{$\pmb F_2$}  
 &
   &   \textbf{SIR} & -2.94 & -2.72 & 0.50 & & -0.68 & -1.77 & 0.50 \\ 
    & &   \textbf{SAVE} & -4.06 & -3.57 & 1.40 & & -3.42 & -3.45 & 1.40 \\
  & &   \textbf{SVR} & -4.08 & -3.87 & 0.90 & & -3.45 & -3.46 & 0.80 \\ 
  \cline{2-10}
 \multirow{3}{*}{$\pmb F_1$}  
 & \multirow{6}{*}{$2\%$} 
   &    \textbf{SIR} & -2.92 & -2.67 & 0.50 & & -0.68 & -1.38 & 0.60 \\ 
    & & \textbf{SAVE} & -3.92 & -3.58 & 1.30 & & -3.08 & -3.01 & 1.40 \\ 
  & &   \textbf{SVR} & -3.91 & -3.61 & 0.90 & & -3.21 & -3.06 & 0.80 \\ 
  \cline{3-10}
 \multirow{3}{*}{$\pmb F_2$}  
 &
   &    \textbf{SIR} & -2.87 & -2.64 & 0.60 & & -0.68 & -1.55 & 0.60 \\ 
    & &   \textbf{SAVE} & -3.64 & -1.98 & 1.40 & & -2.90 & -2.85 & 1.40 \\ 
 & &   \textbf{SVR} & -3.66 & -3.45 & 0.90 & & -3.02 & -2.80 & 0.80 \\ 
  \cline{2-10}
\end{tabular}
\end{table}

The poor performance of SIR in these settings requires a better explanation. In Figure \ref{fig:SIRvsUS_bias} we show graphically how the empirical inverse regression curve may drift away from $v$, resulting in a poor SIR estimate.
On the other hand, the local gradients used by SVR provide good local estimates.
This example shows how methods with higher order statistics are in general more robust to assumptions relaxations.
 
\begin{figure}[h]
\includegraphics[width= .3 \textwidth]{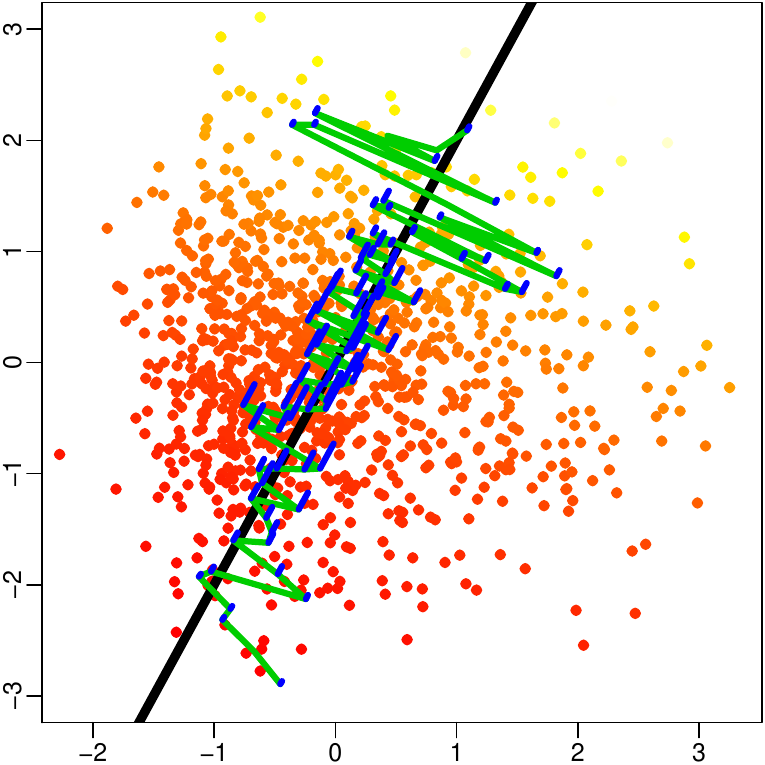}
\includegraphics[width= .3 \textwidth]{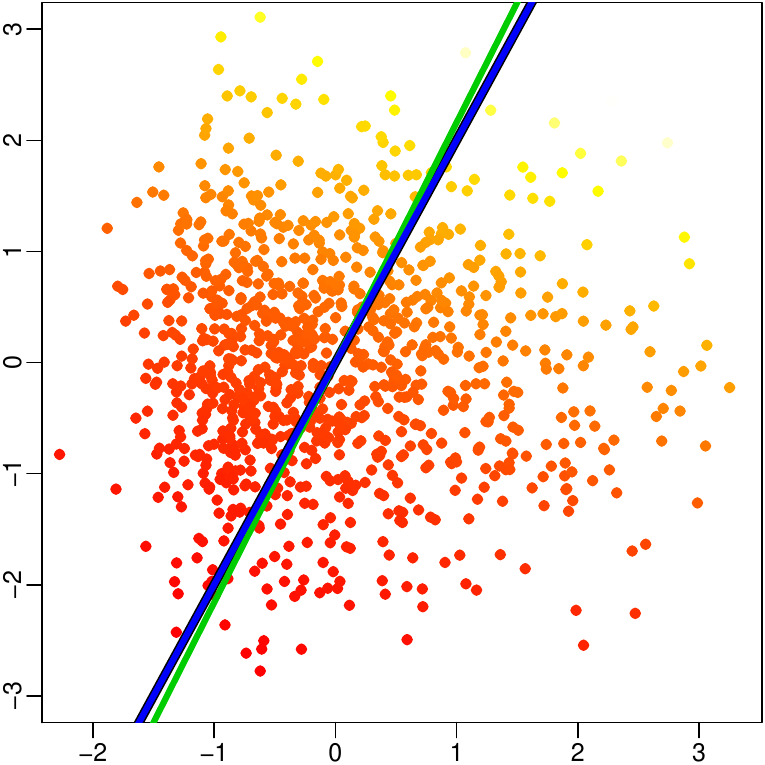}\\
\includegraphics[width= .3 \textwidth]{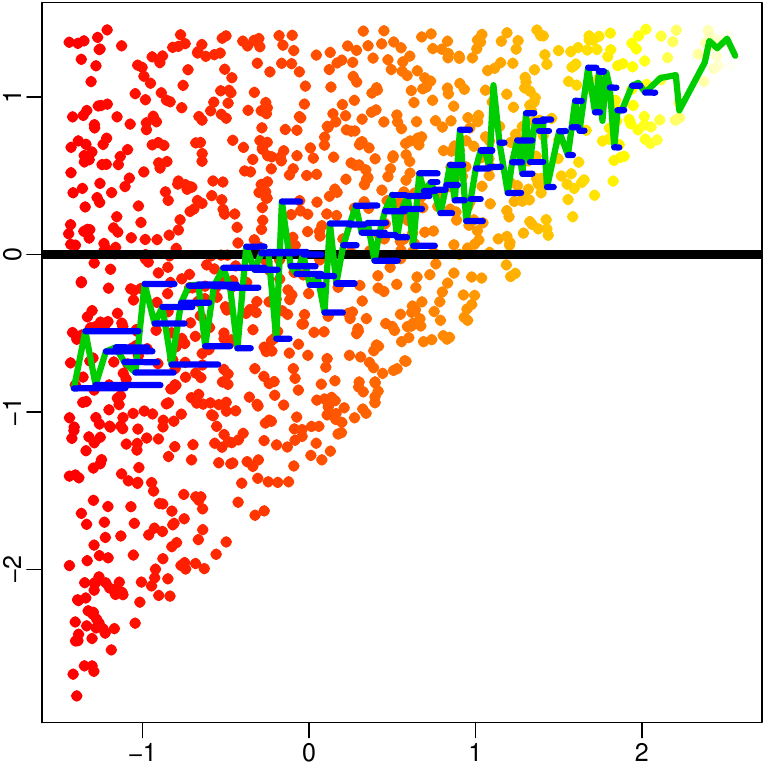}
\includegraphics[width= .3 \textwidth]{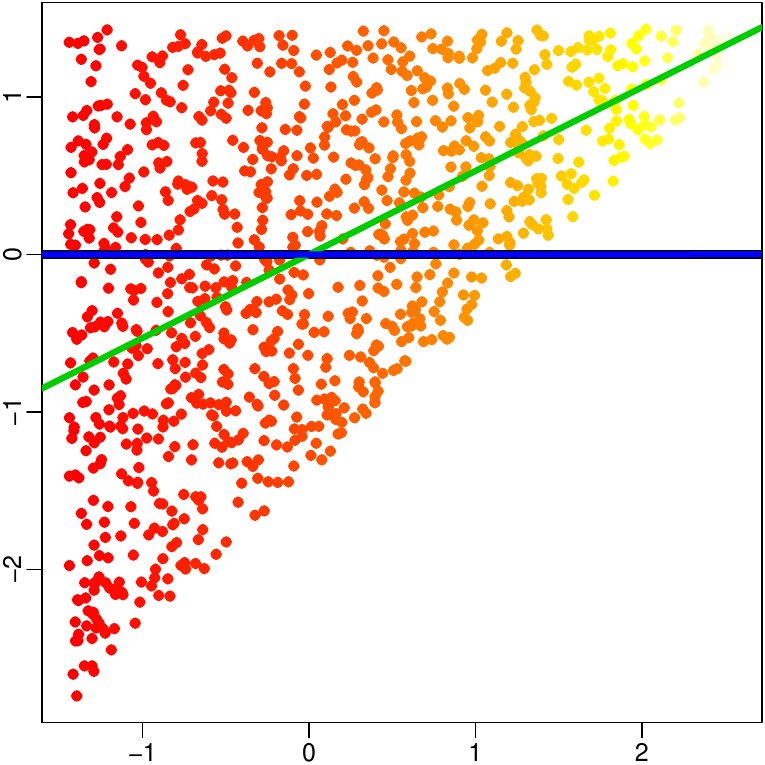}
\caption{\footnotesize Left column displays the ingredients for the estimates: the empirical inverse regression curve (green) used by SIR and the local gradients (blue), with length proportional to the number of samples in the corresponding  level set, used by SVR. Estimates of $v$ using SVR (blue) and SIR (green) are displayed on the right column. The methods are applied on setting S1 (top row) and S2 (bottom row).
The black line indicates $v$, while data points are colored according to the value of the corresponding response variable, generated with $F_2$ and $\sigma=0$, using a red-to-yellow color scale.}
\label{fig:SIRvsUS_bias}
\end{figure}

 \begin{figure}[H]
\hspace{-0.4cm}
\begin{minipage}{0.45\textwidth}
\begin{tikzpicture}
  \node (img)  {\includegraphics[width=.99\textwidth]{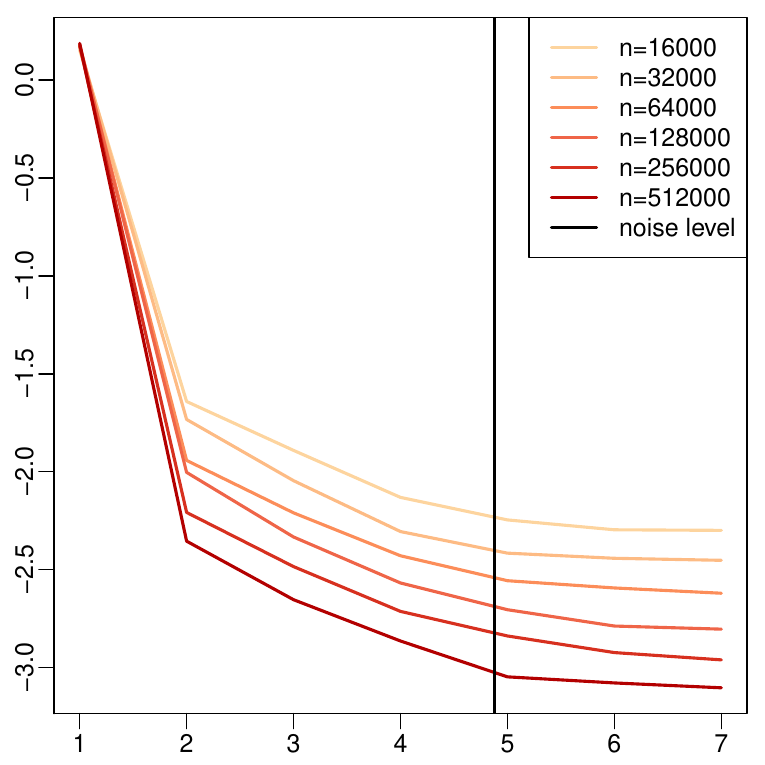}};
  \node[below=of img, node distance=0cm, yshift=1.2cm] {$l$};
  \node[left=of img, node distance=0cm, rotate=90, anchor=center,yshift=-0.9cm] {$\log_{10}(\|\widehat{v}_\el - v \| )$};
 \end{tikzpicture}
\end{minipage}%
\hspace{0.3cm}
\begin{minipage}{0.45\textwidth}
\begin{tikzpicture}
  \node (img)  {\includegraphics[width=.99\textwidth]{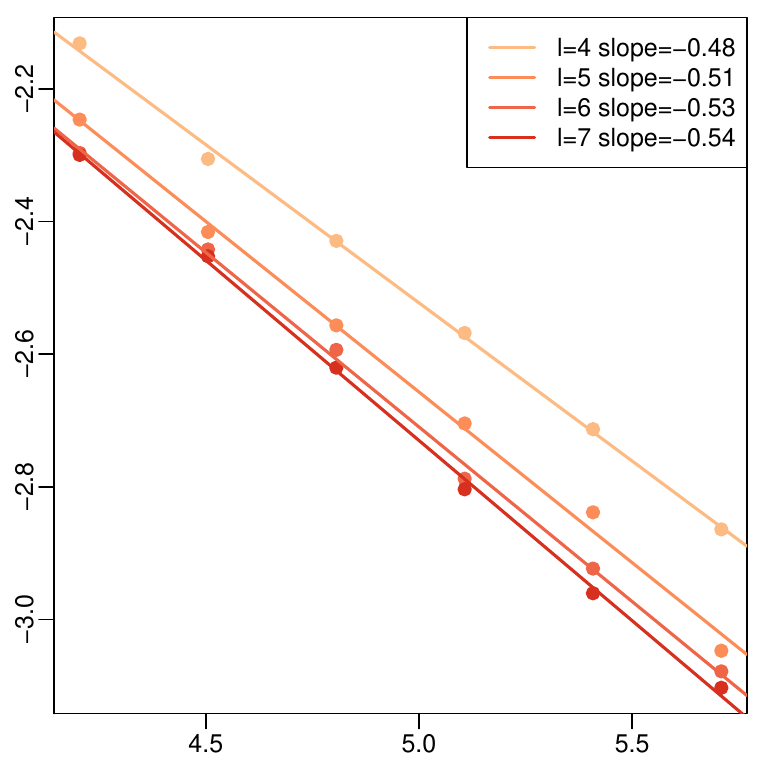}};
\node[below=of img, node distance=0cm, yshift=1.2cm] {$\log_{10}(n)$};
  \node[left=of img, node distance=0cm, rotate=90, anchor=center,yshift=-0.9cm] {$\log_{10}(\|\widehat{v}_\el - v \| )$};
\end{tikzpicture}
\end{minipage}%
\caption{Behavior of the SVR estimate $\widehat{v}_l$ with respect to scale and sample size, for regression of $F_2$ (see text). Left: error versus scale $l$. Right: error versus sample size $n$.}
\label{fig:sim_grad_rates}
\end{figure}
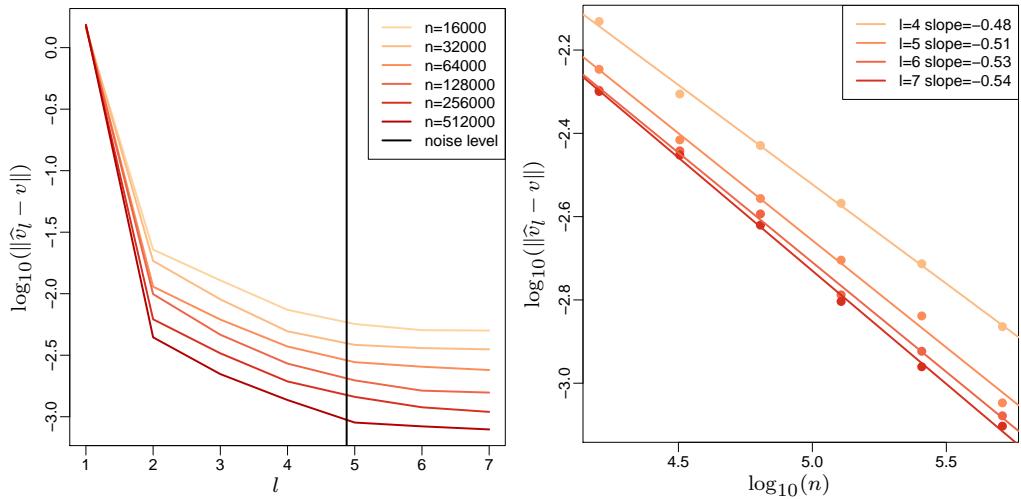

To investigate more extensively the performance of SVR in estimating $v$, we perform another experiment: we draw $X$ from a $10$-dimensional standard normal distribution, and to generate the response variable we use function $F_2$ plus an additive Gaussian noise with standard deviation $\sigma=0.01|f_2(-4)-f_2(4)|$. We repeat the experiment for different values of the sample size $n$. Results are shown in Figure \ref{fig:sim_grad_rates}. The left inset shows that the error in $\widehat v$ stabilizes at scales comparable to the noise level $\sigma$, which suggests that the assumption $ |{\Cijk\el\h}| \gtrsim \sigma $ is needed. The right plot shows that the rate of the error of $\widehat v$, for scales $l$ coarser than the noise level, is approximately $-\frac12$, which is again consistent with Theorem \ref{thm:SVR}.

\subsection{Estimating the regression function}
In this section we perform some experiments to support our theoretical results regarding the regression estimator obtained with SVR.
The first experiment we perform consists on drawing $X_i$, $i=1,...,n$, from a $d$-dimensional standard normal distribution and obtain $Y_i =F_3(X_i)+\zeta_i$ where $\zeta_i \sim \mathcal{N}(0,\sigma^2)$. Here we use function $F_3$ because we want to limit the function smoothness in order to obtain concentration rates comparable with the min-max rate with $s=1$.
We vary the dimension $d=5,10,50,100$, the size of the noise $\sigma$, equal to the $5\%$ and $10\%$ of $|f_3(-4)-f_3(4)|$.
To investigate the convergence rates of the estimator we repeat each experiment for different sample sizes $n$. 
In Figure \ref{fig:sim_rates} we show the empirical MSE, averaged over $10$ repetitions, as a function of the sample size, in logarithmic scale, for both our estimator and the k-Nearest-Neighbor (kNN) regression. We see that the MSE of the SVR estimator decays with a rate slightly better than the optimal value $-2/3$, independently from the dimension $d$ and the noise level $\sigma$: this is all consistent with Theorem \ref{thm:main}. As expected, kNN-regression has a convergence rate which severely deteriorates with the dimension (curse of dimensionality). 
We can also notice that the MSE drops far below the noise level, which confirms the de-noising feature of the SVR estimator.

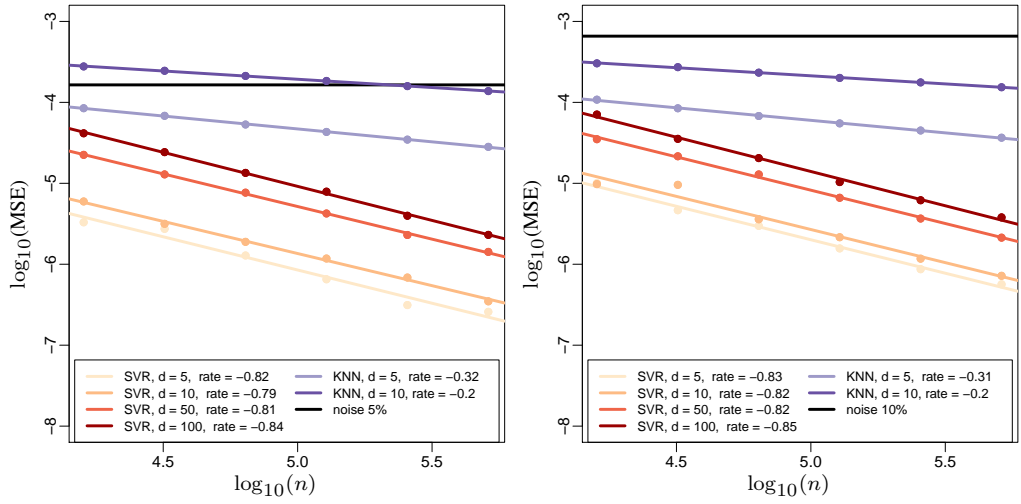
\begin{figure}[h]
\hspace{-0.4cm}
\begin{minipage}{0.45\textwidth}
\begin{tikzpicture}
  \node (img)  {\includegraphics[width=.99\textwidth]{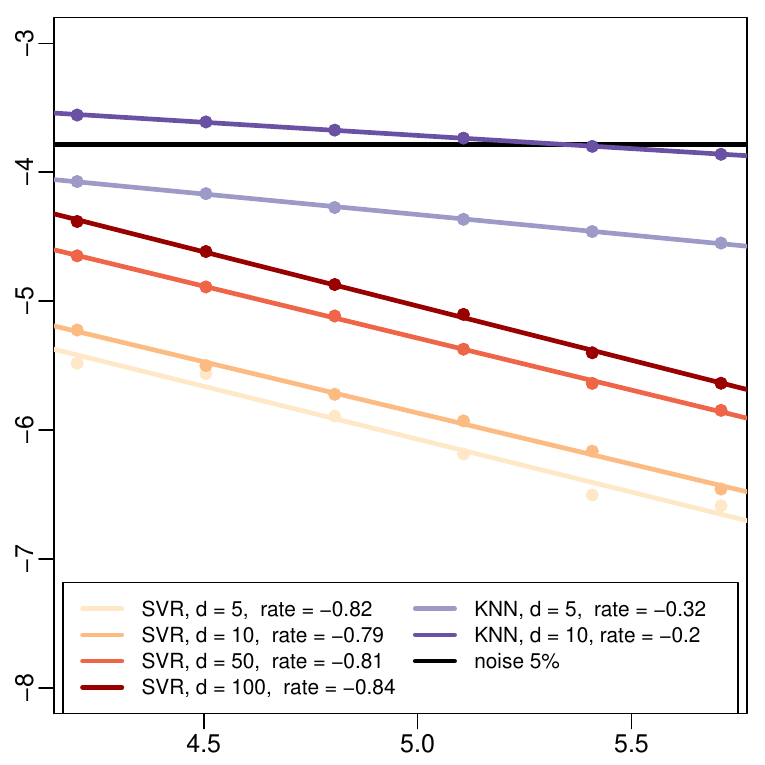}};
  \node[below=of img, node distance=0cm, yshift=1.2cm] {$\log_{10}(n)$};
  \node[left=of img, node distance=0cm, rotate=90, anchor=center,yshift=-0.9cm] {$\log_{10}$(MSE)};
 \end{tikzpicture}
\end{minipage}%
\hspace{0.3cm}
\begin{minipage}{0.45\textwidth}
\begin{tikzpicture}
  \node (img)  {\includegraphics[width=.99\textwidth]{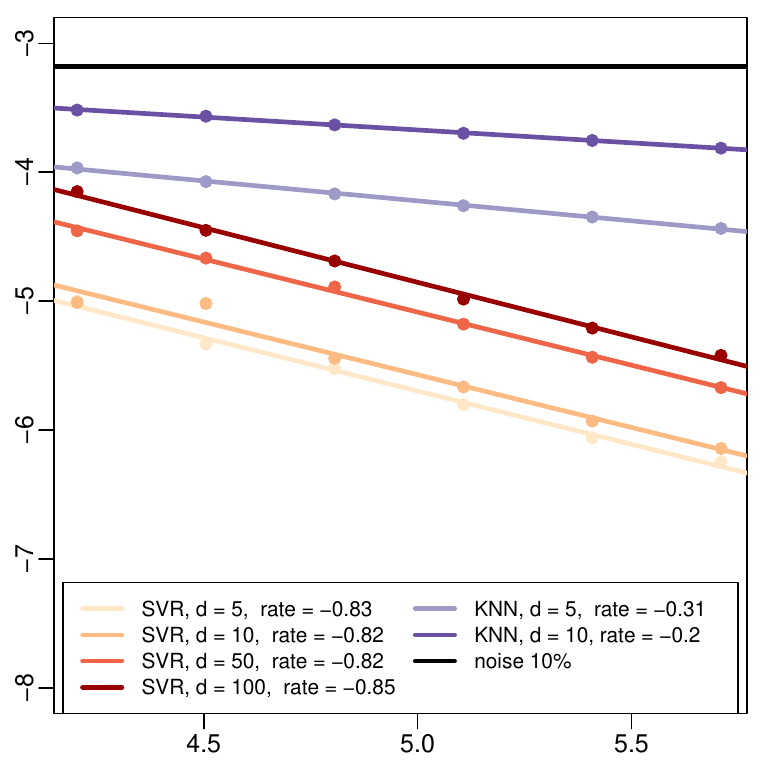}};
\node[below=of img, node distance=0cm, yshift=1.2cm] {$\log_{10}(n)$};
  \node[left=of img, node distance=0cm, rotate=90, anchor=center,yshift=-0.9cm] {$\log_{10}$(MSE)};
\end{tikzpicture}
\end{minipage}%
\caption{Comparison of convergence rates for the regression estimator with SVR and KNN-regression in different settings.}
\label{fig:sim_rates}
\end{figure}

To explore the behavior of the empirical MSE as a function of the scales $l$ and $j$ we conduct another experiment: we draw $X$ from a $10$-dimensional standard normal distribution, and obtain the response variable $Y=F_2(X)+\zeta$, with $\zeta$ Gaussian noise with standard deviation $\sigma=0.01|f_2(-4)-f_2(4)|$.
Figure \ref{fig:heatmaps} shows the behavior of the $\log_{10}(\text{MSE})$, obtained with SVR, for different values of $l$, $j$ and $n$. To obtain robust estimates in regions with high Monte Carlo variability, in regimes where our results do not hold, the errors are averaged over 50 repetition of each setting with a $10\%$ trimming. By observing each row, we notice that the MSE reaches its minimum for low values of $l$ and stays constant for larger $l$. By looking at the plot column-wise, we observe the bias variance trade-off, with coarse scales giving rough estimates, and fine scales resulting in overfitting. As expected, as the sample size grows, the optimal scale $j$ increases.

\begin{figure}[h]
\begin{tikzpicture}
 \node (img) {\includegraphics[width=0.96 \textwidth]{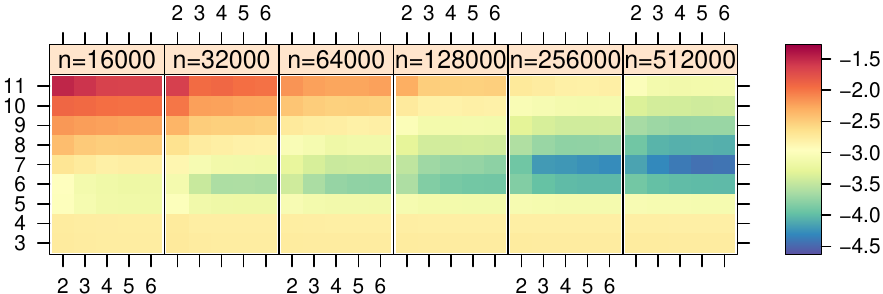}};
 \node[left=of img, node distance=0cm, rotate=90, anchor=center,yshift=-0.9cm] {$j$};
 \node[below=of img, node distance=0cm, yshift=1.2cm] {$\el$};
 \end{tikzpicture}
\caption{Empirical MSE versus sample size $n$ and scales $l$ and $j$.}
\label{fig:heatmaps}
\end{figure}

\paragraph*{Acknowledgements.}This research was partially supported by AFOSR FA9550-17-1-0280, NSF-DMS-1821211, NSF-ATD-1737984. A.L.
acknowledges support from the de Castro Statistics Initiative, Collegio Carlo Alberto, Torino, Italy.
S.V. thanks Timo Klock for the discussion and the useful exchange of views about this and related problems.

\bibliographystyle{plain}
\bibliography{Ref}

\begin{thebibliography}{10}

\bibitem{BachConvexNeural}
F.~Bach.
\newblock Breaking the curse of dimensionality with convex neural networks.
\newblock {\em Journ. of Mach. Learn. Res.}, 19(18):1--53, 2017.

\bibitem{bhatia}
R.~Bhatia.
\newblock {\em Matrix {A}nalysis}, volume 169 of {\em Graduate {T}exts in
  {M}athematics}.
\newblock Springer, 1997.

\bibitem{bickel2007}
P.~J. Bickel and B.~Li.
\newblock Local polynomial regression on unknown manifolds.
\newblock {\em Lecture Notes-Monograph Series}, 54:177--186, 2007.

\bibitem{BCDD2}
P.~Binev, A.~Cohen, W.~Dahmen, and R.~A. DeVore.
\newblock {U}niversal {A}lgorithms for {L}earning {T}heory. {P}art {II}:
  {P}iecewise {P}olynomial {F}unctions.
\newblock {\em Constructive Approximation}, 26(2):127--152, 2007.

\bibitem{BCDDT1}
P.~Binev, A.~Cohen, W.~Dahmen, R.~A. DeVore, and V.~N. Temlyakov.
\newblock {U}niversal {A}lgorithms for {L}earning {T}heory {P}art {I}:
  {P}iecewise {C}onstant {F}unctions.
\newblock {\em Journal of Machine Learning Research}, 6(1):1297--1321, 2005.

\bibitem{Buldygin-Pechuk}
V.~Buldygin and E.~Pechuk.
\newblock Inequalities for the distributions of functionals of sub-gaussian
  vectors.
\newblock {\em Theory of Probability and Mathematical Statistics}, 80:25--36,
  2010.

\bibitem{CAMBANIS1981}
S.~Cambanis, S.~Huang, and G.~Simons.
\newblock On the theory of elliptically contoured distributions.
\newblock {\em Journal of Multivariate Analysis}, 11(3):368--385, 1981.

\bibitem{Carroll97}
R.~J. Carroll, J.~Fan, I.~Gijbels, and M.~P. Wand.
\newblock Generalized partially linear single-index models.
\newblock {\em Journal of the American Statistical Association},
  92(438):477--489, 1997.

\bibitem{Carroll98}
R.~J. Carroll, D.~Ruppert, and A.~H. Welsh.
\newblock Local estimating equations.
\newblock {\em Journal of the American Statistical Association},
  93(441):214--227, 1998.

\bibitem{SAVE}
R.~D. Cook.
\newblock Save: a method for dimension reduction and graphics in regression.
\newblock {\em Communications in Statistics - Theory and Methods},
  29(9-10):2109--2121, 2000.

\bibitem{10.2307/2669934}
R.~D. Cook and H.~Lee.
\newblock Dimension reduction in binary response regression.
\newblock {\em Journal of the American Statistical Association},
  94(448):1187--1200, 1999.

\bibitem{MR3211755}
R.~Coudret, B.~Liquet, and J.~Saracco.
\newblock Comparison of sliced inverse regression approaches for
  underdetermined cases.
\newblock {\em J. SFdS}, 155(2):72--96, 2014.

\bibitem{cui2011}
X.~Cui, W.~K. H\"ardle, and L.~Zhu.
\newblock The {EFM} approach for single-index models.
\newblock {\em Ann. Statist.}, 39(3):1658--1688, 06 2011.

\bibitem{dalalyan2008}
A.~S. Dalalyan, A.~Juditsky, and V.~Spokoiny.
\newblock {A New Algorithm for Estimating the Effective Dimension-Reduction
  Subspace}.
\newblock {\em Journal of Machine Learning Research}, 9:1647--1678, 2008.

\bibitem{dallaglio}
G.~Dall'Aglio.
\newblock Sugli estremi dei momenti delle funzioni di ripartizione doppia.
\newblock {\em Annali della Scuola Normale Superiore di Pisa - Classe di
  Scienze}, Ser. 3, 10(1-2):35--74, 1956.

\bibitem{delbarrio1999}
E.~del Barrio, E.~Gin\'e, and C.~Matr\'an.
\newblock {Central Limit Theorems for the Wasserstein Distance Between the
  Empirical and the True Distributions}.
\newblock {\em Annals of Probability}, 27(2):1009--1071, 1999.

\bibitem{Delecroix1997Efficient}
M.~Delecroix, W.~H\"{a}rdle, and M.~Hristache.
\newblock Efficient estimation in single-index regression.
\newblock Interdisciplinary Research Project 373: Quantification and Simulation
  of Economic Processes. SFB 373 Discussion Paper~37, Humboldt University of
  Berlin, 1997.

\bibitem{delecroix1999}
M.~Delecroix and M.~Hristache.
\newblock M-estimateurs semi-paramétriques dans les modèles à direction
  révélatrice unique.
\newblock {\em Bull. Belg. Math. Soc. Simon Stevin}, 6(2):161--185, 1999.

\bibitem{DELECROIX2006730}
M.~Delecroix, M.~Hristache, and V.~Patilea.
\newblock On semiparametric {M}-estimation in single-index regression.
\newblock {\em Journal of Statistical Planning and Inference}, 136(3):730--769,
  2006.

\bibitem{diaconis1984}
P.~Diaconis and D.~Freedman.
\newblock Asymptotics of graphical projection pursuit.
\newblock {\em Ann. Statist.}, 12(3):793--815, 09 1984.

\bibitem{DuanLi1}
N.~Duan and K.-C. Li.
\newblock Slicing regression: a link-free regression method.
\newblock {\em Ann. Statist.}, 19(2):505--530, 1991.

\bibitem{EATON1986}
M.~L. Eaton.
\newblock A characterization of spherical distributions.
\newblock {\em Journal of Multivariate Analysis}, 20(2):272--276, 1986.

\bibitem{Aggr}
S.~Ga\"iffas and G.~Lecu\'e.
\newblock Optimal rates and adaptation in the single-index model using
  aggregation.
\newblock {\em Electron. J. Statist.}, 1:538--573, 2007.

\bibitem{SILO}
R.~Ganti, N.~Rao, R.~M. Willett, and R.~Nowak.
\newblock Learning single index models in high dimensions.
\newblock arXiv:1506.08910, 2015.

\bibitem{Gyorfi}
L.~Gy{\"o}rfi, M.~Kohler, A.~Krzyzak, and H.~Walk.
\newblock {\em A Distribution-Free Theory of Nonparametric Regression}.
\newblock Springer Series in Statistics. Springer, 2002.

\bibitem{hall1993}
P.~Hall and K.-C. Li.
\newblock On almost linearity of low dimensional projections from high
  dimensional data.
\newblock {\em Ann. Statist.}, 21(2):867--889, 06 1993.

\bibitem{hardle1993}
W.~H\"{a}rdle, P.~Hall, and H.~Ichimura.
\newblock Optimal smoothing in single-index models.
\newblock {\em Ann. Statist.}, 21(1):157--178, 03 1993.

\bibitem{ADE89}
W.~H\"ardle and T.~M. Stoker.
\newblock Investigating smooth multiple regression by the method of average
  derivatives.
\newblock {\em Journal of the American Statistical Association},
  84(408):986--995, 1989.

\bibitem{horowitz1998semiparametric}
J.~L. Horowitz.
\newblock {\em Semiparametric Methods in Econometrics}.
\newblock Lecture Notes in Statistics. Springer New York, 1998.

\bibitem{SA}
M.~Hristache, A.~Juditsky, J.~Polzehl, and V.~Spokoiny.
\newblock {Structure Adaptive Approach for Dimension Reduction}.
\newblock {\em Annals of Statistics}, 29(6):1537--1566, 2001.

\bibitem{hristache2001}
M.~Hristache, A.~Juditsky, and V.~Spokoiny.
\newblock Direct estimation of the index coefficient in a single-index model.
\newblock {\em Ann. Statist.}, 29(3):593--623, 06 2001.

\bibitem{ICHIMURA199371}
H.~Ichimura.
\newblock Semiparametric least squares {(SLS)} and weighted {SLS} estimation of
  single-index models.
\newblock {\em Journal of Econometrics}, 58(1):71--120, 1993.

\bibitem{Isotron}
S.~M. Kakade, V.~Kanade, O.~Shamir, and A.~T. Kalai.
\newblock Efficient learning of generalized linear and single index models with
  isotonic regression.
\newblock {\em Advances in Neural Information Processing Systems 24}, pages
  927--935, 2011.

\bibitem{kalai2009the}
A.~T. Kalai and R.~Sastry.
\newblock The isotron algorithm: High-dimensional isotonic regression.
\newblock In {\em Proceedings of the 22nd Annual Conference on Learning Theory
  (COLT)}, 2009.

\bibitem{Kelker}
D.~Kelker.
\newblock Distribution theory of spherical distributions and a location-scale
  parameter generalization.
\newblock {\em Sankhy\=a: The Indian Journal of Statistics, Series A
  (1961-2002)}, 32(4):419--430, 1970.

\bibitem{10.1214/20-EJS1785}
T.~Klock, A.~Lanteri, and S.~Vigogna.
\newblock {Estimating multi-index models with response-conditional least
  squares}.
\newblock {\em Electronic Journal of Statistics}, 15(1):589 -- 629, 2021.

\bibitem{NIPS2011_4455}
S.~Kpotufe.
\newblock k-nn regression adapts to local intrinsic dimension.
\newblock {\em Advances in Neural Information Processing Systems 24}, pages
  729--737, 2011.

\bibitem{NIPS2013_5103}
S.~Kpotufe and V.~Garg.
\newblock Adaptivity to local smoothness and dimension in kernel regression.
\newblock {\em Advances in Neural Information Processing Systems 26}, pages
  3075--3083, 2013.

\bibitem{kuksin2012}
S.~Kuksin and A.~Shirikyan.
\newblock {\em Mathematics of Two-Dimensional Turbulence}.
\newblock Cambridge Tracts in Mathematics. Cambridge University Press, 2012.

\bibitem{li2018sufficient}
B.~Li.
\newblock {\em Sufficient Dimension Reduction: Methods and Applications with
  R}.
\newblock Chapman \& Hall/CRC Monographs on Statistics and Applied Probability.
  CRC Press, 2018.

\bibitem{DR}
B.~Li and S.~Wang.
\newblock {On Directional Regression for Dimension Reduction}.
\newblock {\em Journal of the American Statistical Association},
  102(479):997--1008, 2007.

\bibitem{CR}
B.~Li, H.~Zha, and F.~Chiaromonte.
\newblock Contour regression: A general approach to dimension reduction.
\newblock {\em The Annals of Statistics}, 33(4):1580--1616, 2005.

\bibitem{SIR}
K.-C. Li.
\newblock Sliced inverse regression for dimension reduction.
\newblock {\em Journal of the American Statistical Association},
  86(414):316--327, 1991.

\bibitem{liao2016learning}
W.~Liao, M.~Maggioni, and S.~Vigogna.
\newblock Learning adaptive multiscale approximations to data and functions
  near low-dimensional sets.
\newblock In {\em 2016 IEEE Information Theory Workshop (ITW)}, pages 226--230.
  IEEE, 2016.

\bibitem{GMRARegression}
W.~Liao, M.~Maggioni, and S.~Vigogna.
\newblock Multiscale regression on unknown manifolds.
\newblock {\em Mathematics in Engineering}, 4(4):1--25, 2022.

\bibitem{SamWan}
R.~J. Samworth, T.~Wang, and Y.~Yu.
\newblock A useful variant of the {D}avis–{K}ahan theorem for statisticians.
\newblock {\em Biometrika}, 102(2):315--323, 2014.

\bibitem{steward1990}
G.~Stewart and Ji{-}guang Sun.
\newblock {\em Matrix Perturbation Theory}.
\newblock Computer Science and Scientific Computing. Academic Press, 1990.

\bibitem{Stoker86consistentestimation}
T.~M. Stoker.
\newblock Consistent estimation of scaled coefficient.
\newblock {\em Econometrica}, pages 1461--1481, 1986.

\bibitem{Stone82}
C.~J. Stone.
\newblock Optimal global rates of convergence for nonparametric regression.
\newblock {\em Ann. Statist.}, 10(4):1040--1053, 12 1982.

\bibitem{vaart1996}
A.~W. van~der Vaart and J.~Wellner.
\newblock {\em Weak Convergence and Empirical Processes: With Applications to
  Statistics}.
\newblock Springer Series in Statistics. Springer, 1996.

\bibitem{Vershynin:NARMT}
R.~Vershynin.
\newblock Introduction to the non-asymptotic analysis of random matrices.
\newblock In Y.~C. Eldar and G.~Kutyniok, editors, {\em Compressed Sensing},
  pages 210--268. Cambridge University Press, 2012.

\bibitem{vershynin_2018}
R.~Vershynin.
\newblock {\em High-Dimensional Probability: An Introduction with Applications
  in Data Science}.
\newblock Cambridge Series in Statistical and Probabilistic Mathematics.
  Cambridge University Press, 2018.

\bibitem{Xia2006}
Y.~Xia.
\newblock Asymptotic distributions for two estimators of the single-index
  model.
\newblock {\em Econometric Theory}, 22(6):1112–1137, 2006.

\bibitem{MAVE}
Y.~Xia, H.~Tong, W.~K. Li, and L.-X. Zhu.
\newblock An adaptive estimation of dimension reduction space.
\newblock {\em Journal of the Royal Statistical Society. Series B (Statistical
  Methodology)}, 64(3):363--410, 2002.

\end{thebibliography}

\newpage
\appendix
\section{Proof of Proposition \ref{prop:hFhv-hFv}} \label{sec:proofsmain}

First, we set out some notation and exclude some low-probability events.
In the case of \ref{P1}, we condition on
$ \| X_i \| \le \ar $ for all $i$'s,
which happens with probability higher than $ 1 - 2n \exp(-\ar^2/2R^2) $
thanks to assumption \ref{X} and Lemma \ref{lem:sub-gauss-tail}.
We define $\rho$ to be the distribution of $X$,
$\rho(\cdot \mid E)$ the conditional distribution of $X$ given $ X \in E $,
and also conditioned on $ \| X_i \| \le \ar $ for all $i$'s in the case of \ref{P1},
and $ \rho_v(\cdot \mid E) $ the push-forward of $ \rho( \cdot \mid E) $
along $ x \mapsto \langle v , x \rangle $.
Let $ u \in \{ v , \hv \} $; when a property is stated for $u$, it is meant to hold for both $v$ and $\hv$.
We write
$$
\bIjkl jku = \{ x \in \R^d : \langle u,x\rangle \in I_{j,k} \} .
$$
We will restrict to the sets $ \bIjkl jku $ with 
\begin{equation} \label{eq:heavysets}
\rho(\bIjkl jku) \gtrsim \eps^2 / \#\K j |f|_{\cc^0}^2 . \tag{K}
\end{equation}
Indeed,
\begin{align*}
& \E_X [ | \hfjv ( \langle \hv , X \rangle )  - \hfjhv ( \langle \hv , X \rangle ) |^2 \II\{ \|X\| \le \ar \} ] \\
= \ & \sum_{k\in\K j} \E_X [ | \hfjv ( \langle \hv , X \rangle )  - \hfjhv ( \langle \hv , X \rangle ) |^2 \II \{ X \in \bIjkl jk\hv \cap B(0,\ar) \} ] ,
\end{align*}
where
\new{, thanks to the fact that the estimators are truncated at $ M = \| F \|_\infty = \| f \|_{\cc^0} $,}
the subsum of the terms in $k$ with $ \rho(\bIjkl jk\hv) \cap B(0,\ar) \lesssim \eps^2 / \#\K j |f|_{\cc^0}^2 $ is bounded by $ \eps^2 $;
hence, we can restrict to $ \rho(\bIjkl jk\hv) \ge \rho(\bIjkl jk\hv \cap B(0,\ar) ) \gtrsim \eps^2 / \#\K j |f|_{\cc^0}^2 $.
If we assume \ref{P1}, then
$
\rho(\bIjkl jkv) = \rho(\bIjkl jk\hv) \gtrsim \eps^2 / \#\K j |f|_{\cc^0}^2
$ as well.
Otherwise, if we assume \ref{P2}, we still have $ \rho ( \bIjkl jkv ) \gtrsim | I_{j,k} | > 1 / \#\K j \gtrsim \eps^2 / \#\K j |f|_{\cc^0}^2 $.

\new{Defining $ n_{j,k|u} = \# \{ X_i \in \bIjkl jku \} $,} we further condition on the event
\begin{equation} \label{eq:enoughsamples}
n_{j,k|u} \gtrsim n\rho(\bIjkl jku) \text{ for all $k$'s} , \tag{$\mathcal{E}_1$}
\end{equation}
which has probability at least $ 1 - C \#\K j \exp(-c \ n \eps^2 / \#\K j |f|_{\cc^0}^2 ) $, thanks to Lemma \ref{lem:hrho-rho} and \eqref{eq:heavysets}.
Also recall that we are conditioning on
\begin{equation} \label{eq:th<2-j}
\| \hv - v \| \le t 2^{-j} . \tag{$\mathcal{E}_2$}
\end{equation}
For two probability measures $\mu$ and $\nu$, we define the Kantorovich distance
$$
 K_\alpha(\mu,\nu) = \sup_{g \in \cc^\alpha, |g|_{\cc^\alpha} \le 1} \int g(x) d(\mu-\nu)(x) .
$$
The proof goes through a series of decompositions into statistics that are defined in Table \ref{tab:stats}
and whose concentration properties are stated in Lemma \ref{lem:stats}.

\begin{table}[H]
\caption{Statistics used in the decompositions of the proof of Proposition \ref{prop:hFhv-hFv}.}
\label{tab:stats}
\ra{2}
\centering
\resizebox{\columnwidth}{!}{%
\begin{tabular}{l l}
 \hline
 $ \widehat{y}_{j,k|u} = \frac{1}{n_{j,k|u}} \sum_i Y_i \II\{X_i \in \bIjkl jku\} $
 \\
 $ \widetilde{y}_{j,k|u} = \frac{1}{n_{j,k|u}} \sum_i F(X_i) \II\{X_i \in \bIjkl jku\} $
 &
 $ \overline{y}_{j,k|u} = \E [ F(X) \mid X \in \bIjkl jku ] $
 \\
 $ \widehat{\zeta}_{j,k|u} = \frac{1}{n_{j,k|u}} \sum_i \zeta_i \II \{ X_i \in \bIjkl jku \} $
 \\
 $ \widehat{x}_{j,k|u} = \frac{1}{n_{j,k|u}} \sum_{i} X_i \II \{ X_i \in \bIjkl jku \} $
 &
 $ \xjk jku = \E [ X \mid X \in \bIjkl jku ] $
 \\
 $ \widehat{\beta}_{j,k|u} = \frac{ \sum_i \langle u , X_i - \widehat{x}_{j,k|u} \rangle \left( Y_i - \widehat{y}_{j,k|u} \right) \II \{ X_i \in \bIjkl jku \} }{ \sum_i | \langle u , X_i - \widehat{x}_{j,k|u} \rangle |^2 \II \{ X_i \in \bIjkl jku \} } $
 \\
\multicolumn{2}{l}{ $ \widehat{q}_{j,k|u} = \frac{1}{n_{j,k|u}} \sum_i \langle u , X_i - \widehat{x}_{j,k|u} \rangle ( F(X_i) - \widetilde{y}_{j,k|u} ) \II \{ X_i \in \bIjkl jku \} $ }
 \\
  $ \widetilde{q}_{j,k|u} = \frac{1}{n_{j,k|u}} \sum_i \langle v , X_i - \overline{x}_{j,k|v} \rangle ( F(X_i) - \overline{y}_{j,k|v} ) \II \{ X_i \in \bIjkl jku \} $
 &
 $ q_{j,k|u} = \E [ \langle v , X - \overline{x}_{j,k|v} \rangle ( F(X) - \overline{y}_{j,k|v} ) \mid X \in \bIjkl jku ] $
 \\
 $ \widehat{s}_{j,k|u} = \frac{1}{n_{j,k|u}} \sum_i | \langle u , X_i - \widehat{x}_{j,k|u} \rangle |^2 \II \{ X_i \in \bIjkl jku \} $
 \\
 $ \widetilde{s}_{j,k|u} = \frac{1}{n_{j,k|u}} \sum_i | \langle v , X_i - \xjk jkv \rangle |^2 \II \{ X_i \in \bIjkl jku \} $
 &
 $ s_{j,k|u} = \Var[ \langle v , X \rangle\mid X \in \bIjkl jku ] $
 \\
 $ \widehat{z}_{j,k|u} = \frac{1}{n_{j,k|u}} \sum_i \langle u , X_i - \widehat{x}_{j,k|u} \rangle \zeta_i \II\{ X_i \in  \bIjkl jku \} $
 \\
 \hline
\end{tabular}%
}
\end{table}

\begin{lem} \label{lem:stats}
Under the assumptions of Proposition \ref{prop:hFhv-hFv}
and adopting the definitions in Table \ref{tab:stats},
for all $k$'s satisfying \eqref{eq:heavysets} and conditioned on \eqref{eq:enoughsamples} and \eqref{eq:th<2-j} we have
\begin{enumerate}[label=\textnormal{(\alph*)},leftmargin=*]
\setlength\itemsep{1em}
\item \label{it:hzeta}
 $ \P \left\{ | \widehat{\zeta}_{j,k|u} | > \tfrac{\eps}{\sqrt{\smash[b]{ \#\K j\rho(\bIjkl jku) }}} \right\} \lesssim \exp(-c \ n \eps^2/ \#\K j \sigma^2 )  $
\item \label{it:hy}
$ \P \left\{ | \widetilde{y}_{j,k|u} - \overline{y}_{j,k|u} | > t^{-1} \tfrac{\eps}{\sqrt{\#\K j\rho(\bIjkl jku)}} \right\} \lesssim \exp(- c \ n \eps^2/ \#\K j t^2 |f|_{\cc^0}^2 ) $
\item \label{it:hq}
 $ | \widehat{q}_{j,k|v} | \le |f|_{\cc^1} \ar^2 2^{-2j} \qquad | \widehat{q}_{j,k|\hv} | \lesssim t |f|_{\cc^1} \ar^2 2^{-2j} \qquad (\alpha = 1) $
 \item \label{it:hs}
 $ \P \{ \widehat{s}_{j,k|u} \lesssim t \ar^2 2^{-2j} \}
 \lesssim \exp(-c \ n \eps^2/ \#\K j t^2 |f|_{\cc^0}^2) $
 \item \label{it:hz}
 $
 \P \left\{ | \widehat{z}_{j,k|u} | \gtrsim t^{-1} \ar 2^{-j} \frac{ \eps }{\sqrt{\#\K j \rho(\bIjkl jku)} } \right\}
 \lesssim \exp(-c \ n \eps^2/ \#\K j t^2 \sigma^2)
 $
 \item \label{it:hx}
 $ \P \left\{ | \langle v , \widehat{x}_{j,k|u} - \xjk jku \rangle | > t^{-1} |f|_{\cc^1}^{-1} \tfrac{\eps}{\sqrt{\#\K j \rho(\Ijk jku)}} \right\} \lesssim \exp\left(-c \ n \eps^2/ \#\K j t^2 |f|_{\cc^1}^2 \ar^2 \right) $
 \item \label{it:tq-q}
 $ \P \{ | \widetilde{q}_{j,k|u} - q_{j,k|u} | > \ar 2^{-j} \tfrac{\eps}{\sqrt{\#\K j \rho(\bIjkl jku)}} \} \lesssim \exp\left( -c \ n \eps^2 / \#\K j t^2|f|_{\cc^1}^2 \ar^2 \right) $
\end{enumerate}
\end{lem}
\begin{proof}
\ref{it:hzeta}. Follows directly from \cite[Proposition 5.10]{Vershynin:NARMT},
exploiting \eqref{eq:heavysets} and \eqref{eq:enoughsamples}.

\ref{it:hy}. By the Bernstein inequality
(along with \eqref{eq:heavysets} and \eqref{eq:enoughsamples}),
we get
\begin{align*}
 & \P \{ | \widetilde{y}_{j,k|u} - \overline{y}_{j,k|u} | > \eps/\sqrt{\#\K j\rho(\bIjkl jku)} \} \lesssim \exp\biggl( -c \frac{ n \rho(\bIjkl jku) \frac{\eps^2}{\#\K j\rho(\bIjkl jku)} }{ |f|_{\cc^0}^2 + |f|_{\cc^0} \frac{\eps}{\sqrt{\#\K j \rho(\bIjkl jku)}} } \biggr) \\
 = \ & \exp\biggl( -c \frac{ n \eps^2 / \#\K j}{ |f|_{\cc^0}^2 + |f|_{\cc^0} \frac{\eps}{\sqrt{\#\K j \rho(\bIjkl jku)}} } \biggr)
 \le \exp(- c \ n \eps^2/ \#\K j |f|_{\cc^0}^2 ) .
 \end{align*}
 
 \ref{it:hq}. Follows by definition and condition \eqref{eq:th<2-j}.

\ref{it:hs}. We decompose
\begin{align*}
| \widehat{s}_{j,k|u} - s_{j,k|v} |
& \le \biggl| \frac{1}{n_{j,k|u}} \sum_i \langle u , \xjk jkv - \widehat{x}_{j,k|u} \rangle \langle u , X_i - \widehat{x}_{j,k|u} \rangle \II \{ X_i \in \bIjkl jku \} \biggr| \\
& + \biggl| \frac{1}{n_{j,k|u}} \sum_i \langle u , X_i - \xjk jkv \rangle \langle u , \xjk jkv - \widehat{x}_{j,k|u} \rangle \II \{ X_i \in \bIjkl jku \} \biggr| \\
& + \biggl| \frac{1}{n_{j,k|u}} \sum_i \langle u - v , X_i - \xjk jkv \rangle \langle u , X_i - \xjk jkv \rangle \II \{ X_i \in \bIjkl jku \} \biggr| \\
& + \biggl| \frac{1}{n_{j,k|u}} \sum_i \langle v , X_i - \xjk jkv \rangle \langle u - v , X_i - \xjk jkv \rangle \II \{ X_i \in \bIjkl jku \} \biggr| \\
& + | \widetilde{s}_{j,k|u} - s_{j,k|u} | \\
& + | s_{j,k|u} - s_{j,k|v} | .
\end{align*}
Using the Bernstein inequality (with \eqref{eq:heavysets} and \eqref{eq:enoughsamples}),
condition \eqref{eq:th<2-j},
and Lemma \ref{lem:W1} we obtain
$ | \widehat{s}_{j,k|u} - s_{j,k|v} | \lesssim t \ar^2 2^{-2j} $,
and hence
$ \widehat{s}_{j,k|u} \gtrsim t \ar^2 2^{-2j} $ by \ref{R},
with probability higher than $ 1 - C \exp(-c \ n \eps^2/ \#\K j t^2 |f|_{\cc^0}^2) $.

\ref{it:hz}. Follows from \cite[Theorem 3.1]{Buldygin-Pechuk}, \eqref{eq:heavysets} and \eqref{eq:enoughsamples}.

\ref{it:hx}. Follows by the Bernstein inequality (with \eqref{eq:heavysets} and \eqref{eq:enoughsamples}).

\ref{it:tq-q}. Follows by the Bernstein inequality (with \eqref{eq:heavysets}, \eqref{eq:enoughsamples}) and condition \eqref{eq:th<2-j}.
\end{proof}

We can now work to establish the main bound of Proposition \ref{prop:hFhv-hFv}.
Let $ x \in B(0,\ar) $, and let $k$ be the index such that $ \langle \hv , x \rangle \in I_{j,k} $.
Then
\begin{align*}
| \hfjhv( \langle \hv , x \rangle ) - \hfjv( \langle \hv , x \rangle ) | 
& \new{ =  | T_M [ \hfjkhv( \langle \hv , x \rangle ) ] - T_M [ \hfjkv( \langle \hv , x \rangle ) ] | } \\
& \new{ \le | \hfjkhv( \langle \hv , x \rangle ) - \hfjkv( \langle \hv , x \rangle ) | } ,
\end{align*}
where
\new{
$ \widehat{f}_{j,k|u} $
defines our local empirical estimator with respect to the oracle ($ u = v $) or the estimated ($ u = \hv $) direction.
In the piecewise constant case, we have
$$
 \widehat{f}_{j,k|u} (t) = \widehat{y}_{j,k|u} ,
$$
while in the piecewise linear case we have
$$
 \widehat{f}_{j,k|u} (t) = \widehat{y}_{j,k|u} + \widehat{\beta}_{j,k|u} ( t - \langle u , \widehat{x}_{j,k|u} \rangle ) .
$$
}
We first approach the piecewise constant case. We have
$$
| \widehat{y}_{j,k|\hv} - \widehat{y}_{j,k|v} | \le | \widehat{\zeta}_{j,k|\hv} |
+ | \widetilde{y}_{j,k|\hv} - \overline{y}_{j,k|\hv} |
+ | \overline{y}_{j,k|\hv} - \overline{y}_{j,k|v} |
+ | \overline{y}_{j,k|v} - \widetilde{y}_{j,k|v} |
+ | \widehat{\zeta}_{j,k|v} | .
$$
Now,
$$
 \P \{ | \widehat{\zeta}_{j,k|u} | > \eps/\sqrt{\smash[b]{ \#\K j\rho(\bIjkl jku) }} \} \lesssim \exp(-c \ n \eps^2/ \#\K j \sigma^2 )
$$
and
$$
\P \{ | \widetilde{y}_{j,k|u} - \overline{y}_{j,k|u} | > \eps/\sqrt{\#\K j\rho(\bIjkl jku)} \} \lesssim \exp(- c \ n \eps^2/ \#\K j |f|_{\cc^0}^2 )
$$
by Lemma \ref{lem:stats}\ref{it:hzeta} and \ref{it:hy}.
In view of \cite[Proposition 1.2.6]{kuksin2012} (bounding $K_\alpha$ in terms of $W_1$) and the Kantorovich--Rubinstein duality for $W_1$, we have
\begin{align*}
 | \overline{y}_{j,k|\hv} - \overline{y}_{j,k|v} | &\le |f|_{\cc^\alpha} K_\alpha ( \rho_v( \cdot \mid \bIjkl jk\hv) , \rho_v(\cdot \mid \bIjkl jkv)) \\
 &\le |f|_{\cc^\alpha}  \left( W_1(\rho_v(\cdot \mid \bIjkl jk\hv),\rho_v(\cdot \mid \bIjkl jkv)) \right)^{\frac{1}{2-\alpha}} \\
 &\lesssim |f|_{\cc^\alpha} \ar^{\frac{1}{2-\alpha}} \| \hv - v \|^{\frac{1}{2-\alpha}} ,
\end{align*}
where in the last inequality we have used Lemma \ref{lem:W1}.
\new{
This provides all the necessary bounds to conclude for the piecewise constant case,
with overall probability obtained by taking the union bound over $\kk_j$.
}

We now take care of the piecewise linear case, for which we can assume $ \alpha = 1 $.
Let us separate the constant and the linear components:
$$
| \hfjkhv( \langle \hv,x\rangle ) - \hfjkv( \langle \hv , x \rangle ) |
\le | \hyjkhv - \hyjkv | +
 |\hmjkhv  \langle \hv , x - \hxjkhv \rangle - \hmjkv ( \langle \hv , x \rangle - \langle v , \hxjkv \rangle ) | .
$$
\new{
For the constant component, we already have the bounds obtained in the piecewise constant case.
We thus focus on the linear component, which we split into
$$
|\hmjkhv  \langle \hv , x - \hxjkhv \rangle - \hmjkv ( \langle \hv , x \rangle - \langle v , \hxjkv \rangle ) | \le \ell_1 + \ell_2 + \ell_3 ,
$$
}
where
\begin{align*}
\ell_1 & = |\hmjkv| |\langle v - \hv , x \rangle| \\
\ell_2 & = |\hmjkhv| | \langle \hv , x - \hxjkhv \rangle - \langle v , x - \hxjkv \rangle | \\
\ell_3 & =  | \hmjkhv - \hmjkv | | \langle v , x - \hxjkv \rangle | .
\end{align*}
We have
\begin{align*}
 & | \widehat{\beta}_{j,k|u} | \le \widehat{s}_{j,k|u}^{-1} \left( | \widehat{q}_{j,k|u} | + | \widehat{z}_{j,k|u} | \right)  \\
 & |\hmjkhv - \hmjkv| \le \hvarjkhv^{-1} | \hcovjkhv - \hcovjkv |
 + \hvarjkhv^{-1} \hvarjkv^{-1} |\hvarjkhv - \hvarjkv| |\hcovjkv| \\
  & \phantom{ |\hmjkhv - \hmjkv| } \hspace{0.85pt} + \hvarjkhv^{-1} | \widehat{z}_{j,k|\hv} |
  + \hvarjkv^{-1} | \widehat{z}_{j,k|v} | . 
\end{align*}
Using Lemma \ref{lem:stats}\ref{it:hs}, \ref{it:hq} and \ref{it:hz}, we get
\begin{align*}
 & | \widehat{\beta}_{j,k|v} | \le |f|_{\cc^1} + t^{-1} \ar^{-1} 2^{j} \frac{\eps}{\sqrt{\#\K j \rho(\bIjkl jk\hv)}} \\
 & | \widehat{\beta}_{j,k|\hv} | \le t |f|_{\cc^1} + t^{-1} \ar^{-1} 2^{j} \frac{\eps}{\sqrt{\#\K j \rho(\bIjkl jk\hv)}} \\
 & |\hmjkhv - \hmjkv| \lesssim t^{-1} \ar^{-2} 2^{2j} | \hcovjkhv - \hcovjkv |
 + t^{-1} \ar^{-2} 2^{2j} |f|_{\cc^1} |\hvarjkhv - \hvarjkv| \\
 & \phantom{ |\hmjkhv - \hmjkv| } \hspace{0.85pt} + t^{-1} \ar^{-1} 2^{j} \frac{\eps}{\sqrt{\#\K j \rho(\bIjkl jk\hv)}}
\end{align*}
 with probability higher than $ 1 - C \exp(-c \ n \eps^2/ \#\K j t^2 |f|_{\cc^0}^2) $.
 Hence
\begin{align*}
 \ell_1 & \le |f|_{\cc^1} \ar \| \hv - v \| + \frac{\eps}{\sqrt{\#\K j \rho(\bIjkl jk\hv)}} \\
 \ell_2 & \le t |f|_{\cc^1} \left( | \langle \hv - v , x \rangle | + | \langle v , \hxjkv - \hxjkhv \rangle | + | \langle v - \hv , \hxjkhv \rangle | \right) \\
 & + t^{-1} \ar^{-1} 2^{j} \frac{\eps}{\sqrt{\#\K j \rho(\bIjkl jk\hv)}} \left( | \langle \hv , x - \hxjkhv \rangle | + | \langle v , x - \hxjkv \rangle | \right) \\
 & \lesssim t |f|_{\cc^1} \left( \ar \| \hv -v \| + | \langle v , \hxjkv - \hxjkhv \rangle | \right) \\
 & + \frac{\eps}{\sqrt{\#\K j \rho(\bIjkl jk\hv)}} \\
 \ell_3 & \le \ar^{-1} 2^{j} | \hcovjkhv - \hcovjkv |
 + |f|_{\cc^1} \ar^{-1} 2^{j} |\hvarjkhv - \hvarjkv|
 + \frac{\eps}{\sqrt{\#\K j \rho(\bIjkl jk\hv)}}
 \end{align*}
 with probability higher than $ 1 - C \exp(-c \ n \eps^2/ \#\K j t^2 |f|_{\cc^0}^2) $.
Now,
\begin{align*}
 | \langle v , \hxjkhv - \hxjkv \rangle |
 \le | \langle v , \hxjkhv - \xjk jk\hv \rangle | + | \langle v , \xjk jk\hv - \xjk jkv \rangle | + | \langle v , \xjk jkv - \hxjkv \rangle | ,
\end{align*}
where, by Lemma \ref{lem:stats}\ref{it:hx},
$$
 | \langle v , \widehat{x}_{j,k|u} - \xjk jku \rangle | \le t^{-1} |f|_{\cc^1}^{-1} \frac{\eps}{\sqrt{\#\K j \rho(\Ijk jku)}}
$$
 with probability higher than  $ 1 - C \exp\left(-c \ n \eps^2/ \#\K j t^2 |f|_{\cc^1}^2 \right) $,
and, thanks to Lemma \ref{lem:W1},
$$
| \langle v , \xjk jk\hv - \xjk jkv \rangle | \le W_1(\rho_v(\cdot \mid \bIjkl jk\hv),\rho_v(\cdot \mid \bIjkl jkv)) \lesssim \ar \| \hv - v \| .
$$

We are now left to estimate $ | \hcovjkhv - \hcovjkv | $ and $ | \hvarjkhv - \hvarjkv | $.
First, we break down \linebreak $ | \hcovjkhv - \hcovjkv | \le \sum_{a=1}^8 Q_a $, where
\begin{align*}
& Q_1 = \biggl| \frac{1}{n_{j,k|\hv}} \sum_i \langle\hv , \xjkv - \hxjkhv \rangle (F(X_i) - \widetilde{y}_{j,k|\hv}) \II\{X_i \in \bIjkl jk\hv\} \biggr| \\
& Q_2 = \biggl| \frac{1}{n_{j,k|\hv}} \sum_i \langle\hv - v,X_i - \xjkv\rangle (F(X_i) - \widetilde{y}_{j,k,\hv}) \II\{X_i \in \bIjkl jk\hv\} \biggr| \\
& Q_3 = \biggl| \frac{1}{n_{j,k|\hv}} \sum_i \langle v,X_i - \xjkv\rangle (\yjkv - \widetilde{y}_{j,k|\hv}) \II\{X_i \in \bIjkl jk\hv\} \biggr| \\
& Q_4 = | \widetilde{q}_{j,k|\hv} - q_{j,k|\hv} | \\
& Q_5 = | q_{j,k|\hv} - q_{j,k|v} | \\
& Q_6 = | q_{j,k|v} - \widetilde{q}_{j,k|v} | \\
& Q_7 = \biggl| \frac{1}{n_{j,k|v}} \sum_i \langle v,\hxjkv - \xjkv\rangle (F(X_i) - \yjkv) \II \{ X_i \in \bIjkl jkv \} \biggr| \\
& Q_8 = \biggl| \frac{1}{n_{j,k|v}} \sum_i \langle v , X_i - \hxjkv \rangle (\widetilde{y}_{jk|v} - \yjkv) \II \{ X_i \in \bIjkl jkv \} \biggr| .
\end{align*}
We bound the terms $T_i$'s as follows.

\noindent$ \pmb Q_1 $.
$$ Q_1 \le | \langle \hv , \xjkv - \hxjkhv \rangle | \frac{1}{n_{j,k|\hv}} \sum_i | F(X_i) - \widetilde{y}_{j,k|\hv} | \II \{ X_i \in \bIjkl jk\hv \} $$
with
$ | F(X_i) - \widetilde{y}_{jk|\hv} | \lesssim t |f|_{\cc^1} \ar 2^{-j} $.
Hence
\begin{align*}
\ar^{-1} 2^j Q_1
& \le t |f|_{\cc^1} | \langle \hv , \hxjkhv - \xjkv \rangle | \\
& \le t |f|_{\cc^1} ( | \langle v , \hxjkhv - \xjkhv \rangle | + | \langle v , \xjkhv - \xjkv \rangle | + \ar \| \hv - v \| ) ,
\end{align*}
where
$$ \P \{ | \langle v , \hxjkhv - \xjkhv \rangle | > t^{-1} |f|_{\cc^1}^{-1} \tfrac{\eps}{\sqrt{\#\K j \rho(\bIjkl jk\hv)}} \} \lesssim \exp\left(-c \ n \eps^2 / \#\K j t^2|f|_{\cc^1}^2 \ar^2 \right) $$
by Lemma \ref{lem:stats}\ref{it:hx},
and
$$ | \langle v , \xjkhv - \xjkv \rangle | \le W_1(\rho_v(\cdot \mid \bIjkl jk\hv),\rho_v(\cdot \mid \bIjkl jkv)) \lesssim \ar \| \hv - v \| $$
by Lemma \ref{lem:W1}.

\noindent$ \pmb Q_2 $.
$ Q_2 \lesssim  2^{-j} t |f|_{\cc^1} \ar^2 \| \hv -v \| $,
hence
$ \ar^{-1} 2^j Q_2 \le t |f|_{\cc^1} \ar \| \hv - v \| $.

\noindent$ \pmb Q_3 $.
$$ Q_3 \le \frac{1}{n_{j,k|\hv}} \sum_i | \langle v,X_i - \xjkv\rangle| |\yjkv - \widetilde{y}_{jk|\hv}| \II \{ X_i \in \bIjkl jk\hv \} $$
with
$ |\langle v,X_i - \xjkv\rangle)| \lesssim t \ar 2^{-j} $.
Hence
$$ \ar^{-1} 2^j Q_3 \le t | \widetilde{y}_{jk|\hv} - \yjkv | \le t (|\widetilde{y}_{jk|\hv} - \yjkhv | + | \yjkhv - \yjkv |) , $$
where
$$ \P \{ |\widetilde{y}_{jk|\hv} - \yjkhv | > t^{-1} \tfrac{\eps}{\sqrt{\#\K j \rho(\Ijkhv)}} \} \le C\exp\left(-c \ n \eps^2 / \#\K j t^2 |f|_{\cc^0}^2 \right) $$
by Lemma \ref{lem:stats}\ref{it:hy},
and
$$ | \yjkhv - \yjkv | \le |f|_{\cc^1} W_1(\rho(\cdot \mid \Ijkhv),\rho(\cdot \mid \Ijkv)) \lesssim |f|_{\cc^1} \ar \| \hv - v \| $$
by Lemma \ref{lem:W1}.

\noindent$ \pmb Q_4 $ and $ \pmb Q_6$.
We apply Lemma \ref{lem:stats}\ref{it:tq-q}.

\noindent$ \pmb Q_5 $.
$$ Q_5 = \left| \int g(\new{z}) (d\rho_v(\new{z} \mid \bIjkl jk\hv) - d\rho_v(\new{z} \mid \bIjkl jkv)) \right| $$
where $ g(\new{z}) = ( \new{z} - \langle v , \xjkv \rangle ) (f(\new{z}) - \yjkv) $
is Lipschitz of constant $\lesssim t |f|_{\cc^1} 2^{-j} \ar$:
\begin{align*}
 | g(\new{z}) - g(s) | \le | \new{z} - s | | f(\new{z}) - \yjkv | + | s - \langle v , \xjkv \rangle | | f(\new{z}) - f(s) |
 \lesssim t |f|_{\cc^1} \ar 2^{-j}  | \new{z} - s | .
\end{align*}
Thus, by Lemma \ref{lem:W1},
$$ \ar^{-1} 2^j Q_5 \lesssim t |f|_{\cc^1} W_1(\rho_v(\cdot \mid \bIjkl jk\hv) - \rho_v(\cdot \mid \bIjkl jkv)) \lesssim t |f|_{\cc^1} \ar \| \hv - v \| . $$

\noindent$ \pmb Q_7 $.
We apply Lemma \ref{lem:stats}\ref{it:hx} on
$ Q_7 \le | \langle v , \hxjkv - \xjkv \rangle | |f|_{\cc^1} \ar 2^{-j} $.

\noindent$ \pmb Q_8 $.
We apply Lemma \ref{lem:stats}\ref{it:hy} on
$ Q_8 \le \ar 2^{-j} | \hyjkv - \yjkv | $.

\new{The quantity $ |\hvarjkhv - \hvarjkv| $ can be estimated analogously through the decomposition. $ |\hvarjkhv - \hvarjkv| \le \sum_{a=1}^9 S_a $, with
\begin{align*}
& S_1 = \biggl| \frac{1}{n_{j,k|\hv}} \sum_i \langle \hv , \xjkv - \hxjkhv \rangle \langle \hv , X_i - \hxjkhv \rangle \II\{X_i \in \bIjkl jk\hv\} \biggr| \\
& S_2 = \biggl| \frac{1}{n_{j,k|\hv}} \sum_i \langle \hv , X_i - \xjkv \rangle \langle \hv , \xjkv - \hxjkhv \rangle \II\{X_i \in \bIjkl jk\hv\} \biggr| \\
& S_3 = \biggl| \frac{1}{n_{j,k|\hv}} \sum_i \langle \hv - v , X_i - \xjkv \rangle \langle \hv , X_i - \xjkv \rangle \II\{X_i \in \bIjkl jk\hv\} \biggr| \\
& S_4 = \biggl| \frac{1}{n_{j,k|\hv}} \sum_i \langle v , X_i - \xjkv \rangle \langle \hv - v , X_i - \xjkv  \rangle \II\{X_i \in \bIjkl jk\hv\} \biggr| \\
& S_5 = | \widetilde{s}_{j,k|\hv} - s_{j,k|\hv} | \\
& S_6 = | s_{j,k|\hv} - s_{j,k|v} | \\
& S_7 = | s_{j,k|v} - \widetilde{s}_{j,k|v} | \\
& S_8 = \biggl| \frac{1}{n_{j,k|v}} \sum_i \langle v , \hxjkv - \xjkv \rangle \langle v , X_i - \xjkv \rangle \II\{X_i \in \bIjkl jkv\} \biggr| \\
& S_9 = \biggl| \frac{1}{n_{j,k|v}} \sum_i \langle v , X_i - \hxjkv \rangle \langle v , \hxjkv - \xjkv \rangle \II\{X_i \in \bIjkl jkv\} \biggr| .
\end{align*}}

\new{We can finally put all the bounds and probabilities together.
Each addend of the decomposition is bounded either by $ t |f|_{\cc^1} \ar \| \hv - v \| $
or by $\eps$,
with probability lower bounded by
$$
1 - C\exp\left(-c \ n \eps^2 / \#\K j t^2 \|f\|_{\cc^1}^2 \ar^2 \right) .
$$
Taking the union bound over $\K j$ completes the proof.}
\qed

\section{Proofs of technical results} \label{sec:proofs}

In our proofs, we make use of the following Lemma to ensure that we have enough local samples,
or to concentrate the empirical measure on the underlying distribution.
\begin{lem} \label{lem:hrho-rho}
Let $X$ be a random variable, and let $ X_1,\dots,X_n $ be independent copies of $X$.
Given a measurable set $E$, define $ \rho(E) = \P\{X\in E\} $ and $ \hrho(E) = n^{-1} \sum_i \mathbbm{1}\{X_i \in E\} $.
Then
$$
\P \{ | \hrho(E) - \rho(E) | > t \} \le 2 \exp \left( - \frac{n t^2/2}{\rho(E) + t/3} \right) .
$$
In particular, for $ t = \rho(E)/2 $ we have
 \begin{align*}
 \P \left\{ \hrho(E) \notin \left[ \frac{1}{2}\rho(E) , \frac{3}{2}\rho(E) \right] \right\} \le \P \left\{ | \hrho(E) - \rho(E) | > \frac{1}{2}\rho(E) \right\} \le 2 \exp\left(-\tfrac{3}{28} n \rho(E)\right) .
 \end{align*}
\end{lem}
\begin{proof}
 The bound follows by a direct application of the Bernstein inequality
 to the random variables $ \mathbbm{1}\{X_i \in E \} $.
\end{proof}

When working with possibly unbounded distributions, we need some control on their tails. A common choice is to assume sub-Gaussian decay.
We recall that a random variable $X$ is sub-Gaussian of variance proxy $R^2$ if
$$ \P \{ |X| > t \} \le 2 \exp \left( - \frac{t^2}{2 R^2} \right) . $$
A random vector $ X \in \R^d $ is sub-Gaussian if $ \langle u,X\rangle $ is sub-Gaussian for every $ u \in \S^{d-1} $.
In particular, bounded and normal distributions are sub-Gaussian.

\begin{lem} \label{lem:sub-gauss-tail}
 Let $ X \in \R^d $ be a sub-Gaussian vector with variance proxy $ R^2 $.
 Then
 $$ \P \{ \|X\| > t \} \le 2 \exp\left( - \frac{t^2}{2dR^2} \right) . $$
\end{lem}
\begin{proof}
Let $ X_k $ be the $k$-th coordinate of $X$. Then
\begin{align*}
 \E \left[ \exp\left( \frac{\| X \|^2}{2 d R^2} \right) \right] = \E \left[ \prod_{k=1}^d \exp\left( \frac{|X_k|^2}{2 d R^2} \right) \right]
\le \left( \prod_{k=1}^d \E \left[ \exp\left( \frac{|X_k|^2}{2 R^2} \right) \right] \right)^{1/d}
\hspace{-10pt} \le 2\, . 
\end{align*}
The result follows from \cite[Proposition 2.5.2]{vershynin_2018}.
\end{proof}


\new{
The lemma below shows that, with high probability,
most samples from a $d+1$-dimensional sub-Gaussian pair of variance proxy $R^2$ fall into a cylinder of radius proportional to $ \sqrt{d} R $ and height proportional to $R$.
\begin{lem} \label{lem:sub-gauss-ball}
Let $ (X_1,Y_1),\dots,(X_n,Y_n) $ be independent copies of a sub-Gaussian pair $ (X,Y) \in \R^{d+1}$ with variance proxy $R^2$.
Then, for every $ \gamma \in (0,1) $,
 $$
 \P \{ \# \{ (X_i,Y_i) \in B(0,C_X\sqrt{d}R) \times [-C_Y R , C_YR] \} < \delta_{X,Y} \gamma n \} \le 2 \exp\left( - \delta_{X,Y} \tfrac{(1-\gamma)^2/2}{1+(1-\gamma)/3} n \right) ,
 $$
 where $ \delta_{X,Y} =  \delta_X \delta_{Y|X} $ with
 $ \delta_X = 1-2 e^{-C_X^2/2} $ and
 $ \delta_{Y|X} = 1 - 2 e^{-C_Y^2 \delta_X / 2} $.
\end{lem}
\begin{proof}
 Let $ B = B(0,C_X \sqrt{d} R) $, $ \Ci = [ - C_Y R , C_Y R ] $
 and $ \rho(B\times\Ci) = \P \{ (X,Y) \in B \times \Ci \} $.
Then
$$
 \rho(B\times\Ci) = \P \{ X \in B \} \P \{ Y \in \Ci \mid X \in B \} .
$$
By Lemma \ref{lem:sub-gauss-tail}, we have
$
\P \{ X \in B \} \ge \delta_X
$
and
$
 \P \{ Y \in \Ci \mid X \in B \} \ge \delta_{Y|X}.
$
Hence,
$$
 \rho(B\times\Ci) \ge \delta_X \delta_{Y|X} = \delta_{X,Y} ,
$$
and therefore
\begin{align*}
 \P \{ n^b < \delta_{X,Y} \gamma n \} \le \P \{ n^b < \rho(B \times \Ci) \gamma n \} .
\end{align*}
An application of Lemma \ref{lem:hrho-rho} with $ t = (1-\gamma)\rho(B \times \Ci) $ gives now
\begin{equation*}
 \P \{ n^b < \rho(B \times \Ci) \gamma n \}
 \le 2 \exp\left( - \rho(B\times \Ci) \tfrac{(1-\gamma)^2/2}{1+(1-\gamma)/3} n \right)
 \le 2 \exp\left( - \delta_{X,Y} \tfrac{(1-\gamma)^2/2}{1+(1-\gamma)/3} n \right) . \qedhere
\end{equation*}
\end{proof}
}

We often carry out the following integration to obtain expectation bounds from bounds in probability.
\begin{lem} \label{lem:P-E}
 Let $ X $ be a random variable.
 Suppose there are $ p \in [1,2] $, $ a \ge e $ and $ b > 0 $ such that $ \P \{ |X| > \eps \} \le a e^{-b\eps^{2p}} $ for every $ \eps > 0 $.
 Then $ \E |X|^2 \le \bigl(\frac{\log a}{b}\bigr)^{1/p} $.
\end{lem}
\begin{proof}
 Integrating over $ \eps > 0 $ we get
  $
  \E |X|^2 \le \int_0^{\eps_0} \eps \ d\eps + \int_{\eps_0}^\infty a e^{-b\eps^{2p}}\eps d\eps
  $
  with $ \eps_0 = (\log a/b)^{1/2p} $.
  The first integral is equal to $ \frac{1}{2} (\log a/b)^{1/p} $,
  while the substitution $ b \eps^{2p} \to \eps $ in the second integral gives
 \[
   \frac{a}{2p} \int_{\log a}^\infty \eps^{1/p - 1} e^{-\eps} d\eps \left(\frac{1}{b}\right)^{1/p}
   \le \frac{a}{2} \int_{\log a}^\infty e^{-\eps} d\eps \left(\frac{1}{b}\right)^{1/p}
   = \frac{1}{2} \left(\frac{1}{b}\right)^{1/p} . \qedhere
  \]
\end{proof}

In the proof of Proposition \ref{prop:hFhv-hFv}
we use the following bound on the 1$^\text{st}$ Wasserstein distance $W_1$
between two conditional distributions.
\begin{lem} \label{lem:W1}
 Let $\rho$ be a probability distribution in $\R^d$,
 $ v , w \in \S^{d-1} $,
 $ \pi_u x = \langle u , x \rangle $
 and $ I \subset \R $ an interval with $ \rho ( \pi_u^{-1}I ) > 0 $ for  $ u \in \{ v , w \} $.
\begin{enumerate}[label=\textnormal{(\alph*)},leftmargin=*]
\item Suppose $\rho$ is spherical. Then
 $$
  W_1 ( \rho ( x \mid \pi_v x \in I ) , \rho ( x \mid \pi_w x \in I ) ) \lesssim \sin (\angle(v,w) ) \E_{X\sim\rho} [ \|X\| \mid \pi_v X \in I ] .
 $$
 \item
Suppose $\rho$ has an upper bounded density
 and $ \rho ( \pi_v^{-1} I ) \gtrsim |I| $,
 and let $\pi_{u\#}$ denote the push-forward.
Then
 $$
 W_1 ( \pi_{v\#} \rho ( x \mid \pi_v x \in I ) , \pi_{v\#} \rho ( x \mid \pi_w x \in I ) ) \lesssim \sin (\angle(v,w) ) \diam(\supp\rho) .
 $$
 \end{enumerate}
\end{lem}
\begin{proof}
Let $X$ be a random vector distributed according to $\rho$, and let $ \theta = \angle(v,w) $.
If $\rho$ is spherical, then
$$
 \rho ( x \mid \pi_w x \in I ) = \rho ( R x \mid \pi_v x \in I )  ,
$$
where $R$ is a rotation of angle $\theta$.
Hence,
\begin{align*}
 W_1 ( \rho ( x \mid \pi_v x \in I ) , \rho ( x \mid \pi_w x \in I ) )
 & = \int \| x - R x \| \ d\rho(x\mid \pi_v x \in I) \\
 & = 2 \sin(\theta/2) \E [ \|X\| \mid \pi_v X \in I ] .
\end{align*}
Assume now the absolutely continuous case.
Let
$$
 \mu = \pi_{v\#} \rho ( x \mid \pi_v x \in I ) ,
 \qquad
 \nu = \pi_{v\#} \rho ( x \mid \pi_w x \in I ) .
$$
By virtue of \cite[Teorema 1]{dallaglio} (\cite[Equation (1.2)]{delbarrio1999}),
we have
$$
 W_1 ( \mu , \nu ) = \int_{-\infty}^{+\infty} | M(t) - N(t) | dt
$$
with
$$
 M(t) = \P \{ \pi_v X \le t \mid \pi_v X \in I \} ,
 \qquad
 N(t) = \P \{ \pi_v X \le t \mid \pi_w X \in I \} .
$$
Let $ J' = \pi_v ( \partial \pi_w^{-1} I \cap \supp\rho) $,
where $\partial$ denotes the boundary, and let $ J = I \setminus J' $.
Then
$$
 \int_{-\infty}^{+\infty} | M(t) - N(t) | dt
 \le | J' |
 + \int_{J} | M(t) - N(t) | dt .
$$
The first term is bounded as
$
| J' | \lesssim \sin(\theta) \diam(\supp\rho)
$.
For the second term, define
$$
 I_u = \pi_u^{-1} I, \qquad I_{u,t} = \pi_v^{-1}(-\infty,t] \cap \pi_u^{-1}I \qquad u \in \{ v,w \} ,
$$
and note that $ \rho(I_{u,t}) \le \rho(I_u) $.
Then
\begin{align*}
 | M(t) - N(t) |
 & = \left | \frac{ \rho( I_{v,t} ) }{ \rho(I_v) }
  - \frac{ \rho( I_{w,t} ) }{ \rho(I_w) } \right |
  \le \frac{ \rho( I_{w} ) | \rho( I_{v,t} ) - \rho( I_{w,t} ) | + | \rho( I_w ) - \rho( I_v ) | \rho(I_{w,t}) }{ \rho(I_v) \ \rho(I_w) } \\
  & \le \rho(I_v)^{-1} \left( | \rho( I_{v,t} ) - \rho( I_{w,t} ) | + | \rho( I_w ) - \rho( I_v ) | \right) .
\end{align*}
Denoting $ A \, \triangle \, B = (A \cup B) \setminus (A \cap B) $, we have
\begin{align*}
 & | \rho( I_{v,t} ) - \rho( I_{w,t} ) | \le \rho( I_{v,t} \ \triangle \ I_{w,t} ) \le \rho( I_{v} \ \triangle \ I_{w} ) , \\
 & | \rho( I_w ) - \rho( I_v ) | \le \rho( I_w \ \triangle \ I_v )
 = \rho( I_v \ \triangle \ I_w )
 \lesssim \sin(\theta) \diam(\supp\rho) .
\end{align*}
The claim now follows from the assumption $ \rho(I_v) \gtrsim |I| \ge |J| $.
\end{proof}

\end{document}